\documentclass[11pt]{article}
\usepackage{amsfonts}
\usepackage{amsmath,amssymb}
\usepackage{amsthm}
\usepackage[mathscr]{euscript}
\usepackage{mathrsfs}
\usepackage{indentfirst}
\usepackage{color}
\usepackage{accents}
\usepackage{enumerate}

\newtheorem{theo}{Theorem}[section]
\newtheorem{lem}{Lemma}[section]

\newtheorem{remark}{Remark}[section]

\numberwithin{equation}{section}

\newcommand{\lbl}[1]{\label{#1}}
\allowdisplaybreaks

\newcommand{\be}{\begin{equation}}
\newcommand{\ee}{\end{equation}}
\newcommand\bes{\begin{eqnarray}} \newcommand\ees{\end{eqnarray}}
\newcommand{\bess}{\begin{eqnarray*}}
\newcommand{\eess}{\end{eqnarray*}}
\newcommand{\bbbb}{\left\{\begin{aligned}}
\newcommand{\nnnn}{\end{aligned}\right.}
\newcommand{\bea}{\begin{align*}}
\newcommand{\eea}{\end{align*}}

\newcommand\ep{\varepsilon}

\newcommand\kk{\left}
\newcommand\rr{\right}
\newcommand\dd{\displaystyle}

\newcommand\dx{{\rm d}x}

\newcommand\lm{\lambda}

\newcommand\yy{\infty}

\newcommand\ol{\overline}
\newcommand\ud{\underline}

\begin{document}\thispagestyle{empty}
\setlength{\baselineskip}{16pt}
\begin{center}
 {\LARGE\bf Dynamics for a diffusive epidemic model with a free boundary: spreading-vanishing dichotomy\footnote{This work was supported by NSFC Grants 11901541, 12101569, 12301247}}\\[4mm]
  {\Large Xueping Li\textsuperscript{\dag}\footnote{Corresponding author. {\sl E-mail}: lixueping@zzuli.edu.cn}, \ \  Lei Li\textsuperscript{\ddag}, \ \  Ying Xu\textsuperscript{\dag},\ \ Dandan Zhu\textsuperscript{\dag}}\\[1.5mm]
\textsuperscript{\dag}{School of Mathematics and Information Science, Zhengzhou University of Light Industry, Zhengzhou, 450002, China}\\
\textsuperscript{\ddag}{College of Science, Henan University of Technology, Zhengzhou, 450001, China}
\end{center}

\date{\today}

\begin{quote}
\noindent{\bf Abstract.} This paper involves a diffusive epidemic model whose domain has one free boundary with the Stefan boundary condition, and one fixed boundary subject to the usual homogeneous Dirichlet or Neumann condition. By using the standard upper and lower solutions method and the regularity theory, we first study some related steady state problems which help us obtain the exact longtime behaviors of solution component $(u,v)$. Then we prove there exists the unique classical solution whose longtime behaviors are governed by a spreading-vanishing dichotomy. Lastly, the criteria determining when spreading or vanishing happens are given with respect to the basic reproduction number $\mathcal{R}_0$, the initial habitat $[0,h_0]$, the expanding rates $\mu_1$ and $\mu_2$ as well as the initial function $(u_0,v_0)$. The criteria reveal the effect of the cooperative behaviors of agents and humans on spreading and vanishing.


\end{quote}

\section{Introduction}
\renewcommand{\thethm}{\Alph{thm}}
{\setlength\arraycolsep{2pt}
Using the reaction-diffusion equations to model the spreading of epidemics is a hot topic in the field of biomathematics. The related research not only unveils some interesting phenomena of propagation, but also promotes the development of the corresponding mathematical theory.
In order to study the propagation of an oral-faecal transmitted epidemic, Hsu and Yang \cite{HY} proposed the reaction-diffusion system
\bes\label{1.1}
\left\{\!\begin{aligned}
&u_{t}=d_1\Delta u-au+H(v), & &t>0,~x\in\mathbb{R},\\
&v_{t}=d_2\Delta v-bv+G(u), & &t>0,~x\in\mathbb{R},
\end{aligned}\right.
 \ees
 where $H(v)$ and $G(u)$ satisfy
\begin{enumerate}
\item[{\bf(H)}]\; $H,G\in C^2([0,\yy))$, $H(0)=G(0)=0$, $H'(z),G'(z)>0$ in $[0,\yy)$, $H''(z), G''(z)<0$ in $(0,\yy)$, and $G(H(\hat z)/a)<b\hat{z}$ for some $\hat{z}>0$.
 \end{enumerate}
In model \eqref{1.1}, $u(t,x)$ and $v(t,x)$ stand for the spatial
concentrations of the bacteria and the infective human population, respectively, at time $t$ and location $x$ in the one dimensional habitat; $-au$ represents the natural death rate of the bacterial population and $H(v)$ denotes the contribution of the infective human to the growth rate of the bacteria; $-bv$ is the fatality rate of the infective human population and $G(u)$ is the infection rate of human population; $d_1$ and $d_2$, respectively, stand for the diffusion rate of bacteria and infective human. They showed that there exists a threshold value $\mathcal{R}_0$, defined by
  \bes
  \mathcal{R}_0=\frac{H'(0)G'(0)}{ab},
  \lbl{x.1}\ees
such that when $\mathcal{R}_0>1$, there exists a $c_*>0$ such that \eqref{1.1} has a positive monotone travelling wave solution if and only if $c\ge c_*$. Moreover, the dynamics of the corresponding ODE system with positive initial value is governed by $\mathcal{R}_0$. More precisely, when $\mathcal{R}_0<1$, $(0,0)$ is globally asymptotically stable; while when $\mathcal{R}_0>1$, there is a unique positive equilibrium $(u^*,v^*)$ that is uniquely given by
  \bes\label{1.2}
  au^*=H(v^*), ~ ~ bv^*=G(u^*),\ees
and is globally asymptotically stable.

If $H(v)=cv$, then system \eqref{1.1} reduces to
\bes\label{1.3}
u_{t}=d_1\Delta u-au+cv, ~ ~ v_{t}=d_2\Delta v-bv+G(u), t>0, ~ x\in\mathbb{R},
 \ees
where $G$ satisfies that $G\in C^2([0,\yy))$, $G(0)=0<G'(u)$ in $[0,\yy)$, $G(u)/u$ is strictly decreasing in $(0,\yy)$ and $\lim\limits_{u\to\yy}{G(u)}/{u}<{ab}/{c}$. The corresponding ODE system was first proposed in \cite{CP} to describe the 1973 cholera epidemic spread in the European Mediterranean regions.

When modeling epidemic, an important issue is to know where the spreading frontier of an epidemic is located, which naturally motivates us to discuss the systems, such as \eqref{1.1} and \eqref{1.3}, on the domain whose boundary is unknown and varies over time, instead of the fixed boundary domain or the whole space. Inspired by the work \cite{DL} where the Stefan boundary condition was incorporated into the mathematical model arising from ecology, Ahn et al \cite{ABL} studied the following variant of \eqref{1.3}
\bes\left\{\!\begin{array}{ll}\label{1.4}
u_t=d_1u_{xx}-au+cv, &t>0,~x\in(g(t),h(t)),\\[1mm]
v_t=-bv+G(u), &t>0,~x\in(g(t),h(t)),\\[1mm]
u(t,x)=v(t,x)=0, &t>0, \; x=g(t){\rm ~ or ~ }h(t),\\[1mm]
h'(t)=-\dd\mu u_x(t,h(t)), ~ g'(t)=-\dd\mu u_x(t,g(t)) & t>0,\\[1mm]
-g(0)=h(0)=h_0>0, ~ u(0,x)=u_{0}(x),~ v(0,x)=v_0(x),&-h_0\le x\le h_0.
 \end{array}\right.
 \ees
They showed \eqref{1.4} has a unique global solution, and its dynamics is governed by a spreading-vanishing dichotomy. More precisely,  as $t\to\yy$, either $[g(t),h(t)]$ converges to a  finite interval and $(u,v)\to(0,0)$, or $[g(t),h(t)]\to\mathbb{R}$ and $(u,v)\to(\tilde u^*,\tilde v^*)$ which is the unique positive constant steady state of \eqref{1.3}. Moreover, the criteria for spreading and vanishing were obtained by using some comparison arguments. Later, when spreading happens, the spreading speed was given by Zhao et al \cite{ZLN} whose arguments rely on a related semi-wave problem.

As we see, the dispersal of $v$ is ignored in \eqref{1.4} since it is assumed that the movement of $v$ is relatively small compared to that of $u$. Wang and Du \cite{WD} supposed that the diffusion of $v$ is also described by random diffusion, and thus studied the following problem
 \bes\left\{\!\begin{array}{ll}\label{1.5}
u_t=d_1u_{xx}-au+cv, &t>0,~x\in(g(t),h(t)),\\[1mm]
v_t=d_2v_{xx}-bv+G(u), &t>0,~x\in(g(t),h(t)),\\[1mm]
u(t,x)=v(t,x)=0, &t>0, \; x=g(t){\rm ~ or ~ }h(t),\\[1mm]
h'(t)=-\dd\mu[u_x(t,h(t))+\rho v_{x}(t,h(t))],& t>0,\\[1mm]
g'(t)=-\dd\mu [u_x(t,g(t))+\rho v_x(t,g(t))] ,& t>0,\\[1mm]
-g(0)=h(0)=h_0>0, ~ u(0,x)=u_{0}(x),~ v(0,x)=v_0(x),&-h_0\le x\le h_0.
 \end{array}\right.
 \ees
They obtained a rather complete understanding for the dynamics of \eqref{1.5}, including a spreading-vanishing dichotomy, criteria for spreading and vanishing as well as the spreading speed when spreading happens. Their results implies that in contrast with model \eqref{1.4}, the introduction of the diffusion of $v$ makes the spreading for \eqref{1.5} more difficult. To say concretely, for model \eqref{1.4}, if
 \[h_0\ge L^*:=\frac{\pi}{2}\sqrt{\frac{d_1b}{cG'(0)-ab}},\]
then spreading happens; for model \eqref{1.5}, if
 \[h_0\ge L^*_1:=\frac{\pi}{2}\sqrt{\frac{d_1b+d_2a+\sqrt{(d_1b+d_2a)^2+4d_1d_2(cG'(0)-ab)}}{2(cG'(0)-ab)}},\]
then spreading occurs for \eqref{1.5}. Clearly, $L^*_1>L^*$. Thus the critical size of initial habitat for \eqref{1.5} is larger than that of \eqref{1.4}.

There are a lot of papers involving the situations where one side of habitat is fixed and the other is free. If the boundary condition at the fixed side is the homogeneous Neumann boundary condition, then the problem is equivalent to the systems with the double free boundaries. If the boundary condition at the fixed side is the homogeneous Dirichlet boundary condition or mixed boundary condition, much different dynamics appear, including longtime behaviors and criteria for spreading and vanishing. For example, please see \cite{KY, Wjde14, Wjde15, WZrwa15, Wcnsns15, Wjfa16, ZWjdde18} and the references therein. Moreover, introducing the Stefan boundary condition to the epidemic model has been attracting much attention over the past decades, a small sample of which can be seen from \cite{LZ,WND,HW2,CLWY,ZRZ}. For recent developments on the applications of free boundary problem for the models from ecology can refer to the expository article \cite{Du}. Inspired by the above works, we consider the free boundary problem of \eqref{1.1}, which takes the form of
\bes\left\{\!\begin{array}{ll}\label{1.6}
u_t=d_1u_{xx}-au+H(v), &t>0,~x\in(0,h(t)),\\[1mm]
v_t=d_2v_{xx}-bv+G(u), &t>0,~x\in(0,h(t)),\\[1mm]
\mathbb{B}[u](t,0)=\mathbb{B}[v](t,0)=u(t,h(t))=v(t,h(t))=0, &t>0,\\[1mm]
h'(t)=-\mu_1 u_x(t,h(t))-\mu_2v_x(t,h(t)), & t>0,\\[1mm]
h(0)=h_0, ~ u(0,x)=u_{0}(x), ~ v(0,x)=v_0(x),&0\le x\le h_0,
 \end{array}\right.
 \ees
where $H$ and $G$ satisfy condition {\bf(H)}, operator $\mathbb{B}[w]=w$ or $w'$, as well as $u_0$ and $v_0$ meet with
   \begin{enumerate}
\item[{\bf(I)}]\; $w\in C^{2}([0,h_0])$, $w'(0)>0$, $w(h_0)=w(0)=0<w(x)$ in $(0,h_0)$ when $\mathbb B[w]=w$, $w(h_0)=w'(0)=0<w(x)$ in $[0,h_0)$ when $\mathbb{B}[w]=w'$.
 \end{enumerate}
The definition of operator $\mathbb{B}$ indicates that the homogeneous Dirichlet or Neumann boundary condition is imposed at the fixed boundary $x=0$, respectively, which brings about quite different dynamics.  Our main results are listed below.

\begin{theo}[Global existence and uniqueness]\label{t1.1}Problem \eqref{1.6} has a unique classical solution $(u,v,h)$ defined in all $t\ge0$, and $(u,v,h)\in [C^{1+\frac{\alpha}{2},2+\alpha}(D^T)]^2\times C^{1+\frac{\alpha}{2}}([0,T])$ for any $T>0$ and $\alpha\in(0,1)$, where $D^T=\{(t,x):0<t\le T, ~ 0\le x\le h(t)\}$. Moreover, there exist some $C_i>0$ with $i=1,2$, depending only on parameters in \eqref{1.6} and $(u_0,v_0)$, such that $0\le u\le C_1$ and $0\le v\le C_2$ for $(t,x)\in[0,\yy)\times[0,h(t)]$.
\end{theo}

By virtue of the strong maximum principle and Hopf lemma for parabolic equations, we see that $u(t,x)>0$, $v(t,x)>0$ and $h'(t)>0$ for $t>0$ and $x\in(0,h(t))$, which implies that $h_{\yy}:=\lim_{t\to\yy}h(t)$ is well defined, and $h_{\yy}\in[h_0,\yy]$. We call the case $h_{\yy}<\yy$ {\it vanishing}, and the case $h_{\yy}=\yy$ {\it spreading}. The longtime behaviors of $(u,v,h)$ are given as follows.

\begin{theo}[Spreading-vanishing dichotomy]\label{t1.2}Let $(u,v,h)$ be the unique solution of \eqref{1.6}. Then one of the following alternatives must happen:

{\it $\underline {Spreading}$:} necessarily $\mathcal{R}_0>1$, $h_{\yy}=\yy$, $\lim_{t\to\yy}(u,v)=(U,V)$ in $C_{\rm loc}([0,\yy))$ if operator $\mathbb{B}[w]=w$, and $\lim_{t\to\yy}(u,v)=(u^*,v^*)$ in $C_{\rm loc}([0,\yy))$ if operator $\mathbb{B}[w]=w'$, where $(u^*,v^*)$ is uniquely given by \eqref{1.2} and $(U,V)$ is the unique bounded positive solution of \eqref{2.8};

{\it $\underline {Vanishing}$:} $h_{\yy}<\yy$, $\lim_{t\to\yy}\|u(t,x)+v(t,x)\|_{C([0,h(t)])}=0$ and $\lambda(h_{\yy})\ge0$, where $\lambda(h_{\yy})$ is the principal eigenvalue of \eqref{2.3}. Moreover, $\lim_{t\to\yy}{\rm e}^{kt}\|u(t,x)+v(t,x)\|_{C([0,h(t)])}=0$ for all  $k\in(0,\lambda(h_{\yy}))$ if $\lambda(h_{\yy})>0$.
\end{theo}

Our next result gives a rather complete description of the criteria for spreading and vanishing.

\begin{theo}\label{t1.3}Let $(u,v,h)$ be the unique solution of \eqref{1.6}. Then the following statements hold.
\begin{enumerate}[$(1)$]
\item If $\mathcal{R}_0\le1$, vanishing happens.
\item Suppose $\mathcal{R}_0>1$. Then there exists a unique critical length $l_0$ for initial habitat $[0,h_0]$ such that spreading happens if $h_0\ge l_0$, where $l_0$ is uniquely given in \eqref{3.3}.
\item Assume that $\mathcal{R}_0>1$, $h_0<l_0$ and $\mu_2=Q(\mu_1)$ with $Q\in C([0,\yy))$, $Q(0)=0$ and strictly increasing to $\yy$. Then there exists a unique $\mu^*_1$ such that spreading happens if and only if $\mu_1>\mu^*_1$.
\item Let $\mathcal{R}_0>1$ and $h_0<l_0$. We parameterize the initial data $(u_0,v_0)=(\tau\vartheta_1,\tau\vartheta_2)$ with $\tau>0$ and $(\vartheta_1,\vartheta_2)$ satisfying {\bf(I)}. Then there exists a unique $\tau^*$ such that spreading happens if and only if $\tau>\tau^*$.
\end{enumerate}
\end{theo}

\begin{remark}\label{r1.1}The spreading speed of \eqref{1.6} will be considered in a separate work. It is expected that when $\mathbb{B}[w]=w$, the spreading speed is non-trivial since the solution component $(u,v)$ converges to a non-constant steady state solution of \eqref{2.8}. More accurate estimates for $(u,v)$, such as \cite[Lemma 2.3]{DMZ} or \cite[Lemma 6.5]{DLou}, are needed. Additionally, motivated by the recent work \cite{WND1}, we intend to obtain some sharp estimates for $(u,v,h)$ when spreading happens in the future.
\end{remark}

This paper is arranged as follows. Section 2 involves some preliminary works, including a comparison principle for free boundary problems, an eigenvalue problem and a steady state problem. Section 3 concerns the dynamics of \eqref{1.6}, consisting of a spreading-vanishing dichotomy and criteria for spreading and vanishing.

\section{Some preliminary works}{\setlength\arraycolsep{2pt}

In this section, we discuss a comparison principle for free boundary problems, an eigenvalue problem and a steady state problem, respectively, which will pave the road for the investigation for the dynamics of \eqref{1.6}. Let us begin with stating a comparison principle, whose proof is similar to \cite[Lemma 2.3]{WD} or \cite[Proposition 3.13]{WND}.

\begin{lem}[Comparison principle]\label{l2.1} Let $(\bar{u},\bar{v},\bar{h})\in [C^{1,2}(\Omega_T)\bigcap C(\overline{\Omega_T})]\times C^1([0,T])$ for $T>0$, and satisfy
\bes\left\{\!\begin{array}{ll}\label{2.1}
\bar u_t\ge d_1\bar u_{xx}-a\bar u+H(\bar v), &0<t\le T,~x\in(0,\bar h(t)),\\[1mm]
\bar v_t\ge d_2\bar v_{xx}-b\bar v+G(\bar u), &0<t\le T,~x\in(0,\bar h(t)),\\[1mm]
\mathbb{B}[\bar u](t,0)\ge0, ~ \mathbb{B}[\bar v](t,0)\ge0, ~ \bar u(t,\bar h(t))\ge0, ~ \bar v(t,\bar h(t))\ge0, &0<t\le T,\\[1mm]
\bar h'(t)\ge-\mu_1 \bar u_x(t,\bar h(t))-\mu_2\bar v_x(t,\bar h(t)), & 0<t\le T,\\[1mm]
\bar h(0)\ge h_0, ~ \bar u(0,x)\ge u_{0}(x), ~ \bar v(0,x)\ge v_0(x),&0\le x\le h_0,
 \end{array}\right.
 \ees
 where $\Omega_T=\{(t,x):0<t\le T, ~ 0<x<\bar{h}(t)\}$. Then the unique solution $(u,v,h)$ of \eqref{1.6} satisfies
 \[h(t)\le \bar{h}(t), ~ u(t,x)\le \bar{u}(t,x),~ v(t,x)\le \bar{v}(t,x) ~ {\rm for ~ }t\in[0,T], ~ x\in[0,h(t)].\]
\end{lem}

We usually call $(\bar{u},\bar{v},\bar{h})$ in the above lemma an upper solution for \eqref{1.6}. If we reverse all the inequalities in \eqref{2.1}, then we can define a lower solution. Moreover, from Lemma \ref{l2.1}, it follows that the unique solution $(u,v,h)$ of \eqref{1.6} is strictly increasing with respect to the initial functions $u_0$ and $v_0$, as well as $\mu_i$ for $i=1,2$.

Consider the following eigenvalue problem
\bes\label{2.2}
-\omega''=\nu\omega(x) , ~ ~ x\in(0,l); ~ ~ ~ \mathbb{B}[\omega](0)=\omega(l)=0.
\ees
Simple calculations show that the unique principal eigenpair $(\nu_1,\omega)$ of \eqref{2.2} is given by
 \bess\left\{\!\begin{array}{ll}
\dd\nu_1=\frac{\pi^2}{l^2}, ~ ~ \omega=\sin\frac{\pi}{l}x ~ ~ ~ {\rm if ~ }\mathbb{B}[\omega](0)=\omega(0)=0,\\
\dd\nu_1=\frac{\pi^2}{4l^2}, ~ ~ \omega=\cos\frac{\pi}{2l}x ~ ~ ~ {\rm if ~ }\mathbb{B}[\omega](0)=\omega'(0)=0.
 \end{array}\right.
 \eess

 Let $l>0$. We study the eigenvalue problem
 \bes\left\{\!\begin{array}{ll}\label{2.3}
-d_1\phi''+a\phi-H'(0)\psi=\lambda \phi, &x\in(0,l),\\[1mm]
 -d_2\psi''+b\psi-G'(0)\phi=\lm\psi, &x\in(0,l),\\[1mm]
\mathbb{B}[\phi](0)=\mathbb{B}[\psi](0)=\phi(l)=\psi(l)=0.
 \end{array}\right.
 \ees
\begin{lem}\label{l2.2}The eigenvalue problem \eqref{2.3} has a unique eigenvalue $\lambda$ with a positive eigenfunction $(\phi,\psi)$. More precisely,
\bes\left\{\!\begin{array}{ll}\label{2.4}
\dd\lambda=\frac {d_1\nu_1+a+d_2\nu_1+b-\sqrt{\big(d_1\nu_1+a-d_2\nu_1-b\big)^2
 +4G'(0)H'(0)}}2,\\[3mm]
 \dd\phi=\frac{H'(0)}{d_1\nu_1+a-\lambda}\sin\frac{\pi}{l}x, ~ ~ \psi=\sin\frac{\pi}{l}x ~ ~ ~ {\rm if ~ operator ~ }\mathbb{B}[w]=w,\\[4mm]
 \dd\phi=\frac{H'(0)}{d_1\nu_1+a-\lambda}\cos\frac{\pi}{2l}x, ~ ~ \psi=\cos\frac{\pi}{2l}x ~ ~ ~ {\rm if ~ operator ~ }\mathbb{B}[w]=w',
 \end{array}\right.
\ees
where $\nu_1$ is the principal eigenvalue of \eqref{2.2}.
\end{lem}
\begin{proof}Recall that $(\nu_1,\omega)$ is the unique eigenpair of \eqref{2.2}.
Set $\phi(x)=p\omega(x)$ and $\psi(x)=\omega(x)$.
Substituting such $(\phi,\psi)$ into \eqref{2.3} yields
 \bess\left\{\!\begin{array}{ll}
d_1\nu_1 p\omega+ap\omega-H'(0)\omega=\lambda p\omega, &x\in(0,l),\\[1mm]
 d_2\nu_1 \omega+b\omega-G'(0)p\omega=\lm \omega, &x\in(0,l),
 \end{array}\right.
 \eess
which is equivalent to the following algebraic eigenvalue problem
\bess\left\{\!\begin{array}{ll}
d_1\nu_1p+ap-H'(0)=\lambda p, \\
 d_2\nu_1+b-G'(0)p=\lm .
 \end{array}\right.
 \eess
Since $H'(0), G'(0)>0$, it is not hard to show that the above algebraic eigenvalue problem has two real eigenvalues $\lm_{\pm}$
 \bess
 \lm_{\pm}=\frac {d_1\nu_1+a+d_2\nu_1+b\pm\sqrt{\big(d_1\nu_1+a-d_2\nu_1-b\big)^2
 +4G'(0)H'(0)}}2.
 \eess
It is easy to show that only the eigenvector $(p, 1)^T$ corresponding to $\lm_-$ is positive, i.e., $p>0$. Moreover, direct computation shows $p=\frac{H'(0)}{d_1\nu_1+a-\lambda_-}$. So \eqref{2.3} has an eigenpair defined in \eqref{2.4}.

Now we prove that $(\lambda,\phi,\psi)$ defined in \eqref{2.4} is the unique eigenpair of $\eqref{2.3}$ with a positive eigenfunction. Assume on the contrary that  $(\lambda_1,\Phi,\Psi)$ is another eigenpair of \eqref{2.3} with a positive eigenfunction. By \eqref{2.3} and integrating by parts, we have
\bess
\lambda\int_{0}^{l}\phi\Phi\dx&=&\int_{0}^{l}\bigg(-d_1\phi_{xx}\Phi+a\phi\Phi-H'(0)\psi\Phi\bigg)\dx\\
&=&\int_{0}^{l}\bigg(-d_1\phi\Phi_{xx}+a\phi\Phi-H'(0)\psi\Phi\bigg)\dx\\
&=&\int_{0}^{l}\bigg(H'(0)\phi\Psi+\lambda_1\phi\Phi-H'(0)\psi\Phi\bigg)\dx,
\eess
which implies
\bes\label{2.5}
(\lambda-\lambda_1)\int_{0}^{l}\phi\Phi\dx=H'(0)\int_{0}^{l}(\phi\Psi-\psi\Phi)\dx.
\ees
Similarly, we can derive
\bess
(\lambda-\lambda_1)\int_{0}^{l}\psi\Psi\dx=G'(0)\int_{0}^{l}(\psi\Phi-\phi\Psi)\dx,
\eess
which, combined with \eqref{2.5}, leads to
 \bess
(\lambda-\lambda_1)\kk(\frac{\int_{0}^{l}\phi\Phi\dx}{H'(0)}
 +\frac{\int_{0}^{l}\psi\Psi\dx}{G'(0)}\rr)=0.
 \eess
Since both $\phi$, $\psi$, $\Phi$ and $\Psi$ are positive in $(0,l)$, we obtain $\lambda=\lambda_1$. The uniqueness follows. The proof is finished.
\end{proof}

From Lemma \ref{l2.2}, we immediately derive the following results.
\begin{enumerate}[$(1)$]
\item The principal eigenvalue $\lambda$ of \eqref{2.3} is strictly decreasing with respect to $l>0$.
\item $\lambda\to\frac{1}{2}\big(a+b-\sqrt{(a-b)^2+4G'(0)H'(0)}\big)$ as $l\to\yy$.
\item $\lambda\to\yy$ as $l\to0$.
\end{enumerate}
Moreover, it is easy to see that $a+b-\sqrt{(a-b)^2+4G'(0)H'(0)}<0$ is equivalent to $\mathcal{R}_0>1$. Rewrite $\lambda$ as $\lambda(l)$ to stress the dependence on $l$. Hence if $\mathcal{R}_0>1$, then there exists a unique $l_0>0$ such that $\lambda(l_0)=0$ and $\lambda(l)(l-l_0)<0$ if $l\neq l_0$.

Next we consider the following steady state problem
\bes\left\{\!\begin{array}{ll}\label{2.6}
-d_1u''=-au+H(v), \;\;&x\in(0,l),\\[1mm]
-d_2v''=-bv+G(u), &x\in(0,l),\\[1mm]
\mathbb{B}[u](0)=\mathbb{B}[v](0)=u(l)=v(l)=0.
 \end{array}\right.
\ees
\begin{lem}\label{l2.3}Let $\mathcal{R}_0>1$ and $\lambda(l)$ be the principal eigenvalue of \eqref{2.3}. Then the following statements are valid.
\begin{enumerate}[$(1)$]
\item Problem \eqref{2.6} has a unique positive solution $(u_l,v_l)$ if and only if $\lambda(l)<0$, i.e., $l>l_0$. Moreover, $0<u_l<u^*$, $0<v_l<v^*$ in $(0,l)$ and $(u_l,v_l)$ is globally asymptotically stable where $(u^*,v^*)$ is uniquely given by \eqref{1.2}.
\item Suppose $l>l_0$. Then $(u_l,v_l)$ is strictly increasing in $l>l_0$. Additionally,
\bes\left\{\!\begin{array}{ll}\label{2.7}
(u_l,v_l)\to(U,V) ~{\rm in ~ }C^2_{\rm loc}([0,\yy))  ~ {\rm as ~ }l\to\yy ~ ~ {\rm when~ } \mathbb{B}[w]=w,\\
(u_l,v_l)\to(u^*,v^*) ~{\rm in ~ }C^2_{\rm loc}([0,\yy))  ~ {\rm as ~ }l\to\yy ~ ~ {\rm when ~ }\mathbb{B}[w]=w',\\
 \end{array}\right.
\ees
where $(u^*,v^*)$ is uniquely given by \eqref{1.2} and $(U,V)$ is the unique bounded positive solution of
\bes\left\{\!\begin{array}{ll}\label{2.8}
-d_1u''=-au+H(v), \;\;&x\in(0,\yy),\\[1mm]
-d_2v''=-bv+G(u), &x\in(0,\yy),\\[1mm]
u(0)=v(0)=0.
 \end{array}\right.
\ees
\end{enumerate}
\end{lem}
\begin{proof}(1) Assume that \eqref{2.6} has a positive solution $(u_l,v_l)$. Due to condition {\bf(H)}, we have
\bess\left\{\!\begin{array}{ll}
-d_1u''+au-H'(0)v<0, \;\;&x\in(0,l),\\[1mm]
-d_2v''+bv-G'(0)u<0, &x\in(0,l),\\[1mm]
\mathbb{B}[u](0)=\mathbb{B}[v](0)=u(l)=v(l)=0.
 \end{array}\right.
\eess
Then arguing as in the proof of uniqueness in Lemma \ref{l2.2}, we immediately obtain $\lambda(l)<0$.

Suppose $\lambda(l)<0$. Now we show the existence and uniqueness of positive solution of \eqref{2.6}.  We only handle the Dirichlet boundary condition $\mathbb{B}[w]=w$ since for the Neumann boundary condition $\mathbb{B}[w]=w'$, the desired results can be obtained by following the similar lines (actually, it is simpler).

 Let $(\phi,\psi)$ be the positive eigenfunction of $\lambda(l)$ with $\|\phi+\psi\|_{C([0,l])}=1$. Define $(\underline{u},\underline{v})=(\ep\phi,\ep\psi)$ and $(\bar{u},\bar{v})=(u^*,v^*)$ with $\ep$ to be determined later. Clearly, $(\bar{u},\bar{v})$ is an upper solution of \eqref{2.6}.

Simple calculations yield that for $x\in(0,l)$,
\bess
 -d_1\underline{u}''(x)+a\ud{u}-H(\ud{v})&=&\ep\big[-d_1\phi''(x)+a\phi-H(\ep\psi)/\ep\big]=\ep\big[\lm(l)\phi+H'(0)\psi-H(\ep\psi)/\ep\big].
 \eess
 Define
 \[F(x)=\lm(l)\phi(x)+H'(0)\psi(x)-H(\ep\psi(x))/\ep, ~ ~ x\in(0,l).\]
 We now show there exists a small $\ep_0>0$ such that $F(x)<0$ in $(0,l)$ if $\ep<\ep_0$. By the Hopf boundary Lemma, we have $\phi'(0)>0$, $\phi'(l)<0$, $\psi'(0)>0$ and $\psi'(l)<0$, which implies that $F'(0)=\lambda(l)\phi'(0)<0$ and $F'(l)=\lambda(l)\phi'(l)+H'(0)\psi'(l)-H'(\ep\psi(l))\psi'(l)<0$ if $\ep>0$ is small enough. So there is a small $\delta>0$ such that $F(x)<0$ in $(0,\delta)\cup(l-\delta,l)$. For $x\in[\delta,l-\delta]$,
 we have
 \bess
 F(x)\le\lambda(l)\min_{x\in[\delta,l-\delta]}\phi(x)+H'(0)-H(\ep)/\ep<0,
 \eess
 provided that $\ep$ is sufficiently small. Therefore, we derive that there exists a small $\ep_0>0$ such that if $\ep<\ep_0$, then for $x\in(0,l)$, $-d_1\underline{u}''(x)+a\ud{u}-H(\ud{v})<0$.

 Similarly, we can prove that there is a small $\ep_1>0$ such that if $\ep<\ep_1$, then for $x\in(0,l)$, $-d_2\underline{v}''(x)+b\ud{v}-G(\ud{u})<0$.
 Let $\ep\in(0,\min\{\ep_0,\ep_1\})$ satisfying $\ep\phi<u^*$ and $\ep\psi<v^*$ in $[0,l]$. It is thus clear that $(\underline{u},\underline{v})$ and $(\bar{u},\bar{v})$ are the ordered upper and lower solution of \eqref{2.6}. Then by the upper and lower solutions method for elliptic systems, we obtain that \eqref{2.6} has at least one positive solution $(u,v)$ with $(\underline{u},\underline{v})\le(u,v)\le(u^*,v^*)$.

 Now we show the uniqueness. Let $(u_1,v_1)$ and $(u_2,v_2)$ be the positive solution of \eqref{2.6}. By Hopf boundary lemma, we see that $u'_i(0)>0$, $u'_i(l)<0$, $v'_i(0)>0$ and $v'_i(l)<0$. Thus there exists $k\ge1$ such that $k(u_1,v_1)\ge(u_2,v_2)$  for $x\in[0,l]$. Define
 \[k^*=\inf\{k\ge1: k(u_1,v_1)\ge(u_2,v_2) {\rm ~ in ~ }[0,l]\}.\]
 Clearly, $k^*\ge1$ is well defined and $k^*(u_1,v_1)\ge (u_2, v_2)$ in $[0,l]$. We claim that $k^*=1$. Assume on the contrary that $k^*>1$. Direct computations show that
 \bess
 -d_1(k^*u_1-u_2)''+a(k^*u_1-u_2)
 =k^*H(v_1)-H(v_2)>H(k^*v_1)-H(v_2)\ge 0,\;\;x\in(0,l),
 \eess
and $k^*u_1(0)-u_2(0)=0$, $k^*u_1(l)-u_2(l)=0$. By the strong maximum principle and the Hopf lemma for elliptic equations, we have $k^*u_1-u_2>0$ in $(0,l)$,  $(k^*u_1-u_2)'(0)>0$ and $(k^*u_1-u_2)'(l)<0$. Therefore, we can find a $\ep_0>0$ satisfying $k^*/(1+\ep_0)>1$ such that $k^*u_1-u_2>\ep_0 u_2$, i.e., $\frac {k^*}{1+\ep_0}u_1\ge u_2$ in $[0,l]$. Similarly, we can show there exists a small $\ep_1>0$ with $k^*/(1+\ep_1)>1$ such that $\frac {k^*}{1+\ep_1}v_1\ge v_2$ in $[0,l]$. This clearly contradicts the definition of $k^*$. Thus $k^*=1$, i.e., $(u_1,v_1)\ge(u_2,v_2)$ in $[0,l]$. By exchanging the position of $(u_1,v_1)$ and $(u_2,v_2)$, we analogously derive $(u_2,v_2)\ge(u_1,v_1)$ in $[0,l]$. Therefore, the uniqueness follows.

By the strong maximum principle, we easily obtain that the unique positive solution $(u_l,v_l)$ of \eqref{2.6} satisfies $0<u_l<u^*$, $0<v_l<v^*$ in $(0,l)$.

Now we prove the stability. Let $(u_0,v_0)\in [L^\yy(0,l)]^2$, $u_0,v_0\ge 0$ and $u_0+v_0\not\equiv 0$. It is easy to see that the following initial boundary value problem
\bes\left\{\!\begin{array}{ll}\label{2b.4}
u_t=d_1u_{xx}-au+H(v), &t>0,~x\in(0,l),\\[1mm]
v_t=d_2v_{xx}-bv+G(u), &t>0,~x\in(0,l),\\[1mm]
u(t,0)=v(t,0)=u(t,l)=v(t,l)=0, &t>0, \; i=1,2,\\[1mm]
u(0,x)=u_{0}(x), ~ v(0,x)=v_0(x), &0\le x\le l
 \end{array}\right.
 \ees
has a unique global solution $(u(t,x), v(t,x))$. By the interior estimates for parabolic equations, we have $(u,v)\in [C^{1,2}((0,\yy)\times[0,l])]^2$. In view of the strong maximum principle and the Hopf boundary for parabolic equations, we see $u,v>0$ in $(0,\yy)\times(0,l)$ and $u_x(t,0)>0$, $v_x(t,0)>0$, $u_x(t,l)<0$ and $v_x(t,l)<0$ for all $t>0$. Thus $u(1,x),v(1,x)\in C^2([0,l])$, $u_x(1,0)>0$, $v_x(1,0)>0$, $u_x(1,l)<0$ and $v_x(1,l)<0$. Recall $(\phi,\psi)$ is the positive eigenfunction of $\lambda(l)$ and $(u_l,v_l)$ is the unique positive solution of \eqref{2.6}. As above, we can find $0<\ep\ll 1$ and $M\gg1$ such that $(u(1,x), v(1,x))\ge \ep(\phi,\psi)$ and $(u(1,x),v(1,x))\le M(u_l,v_l)$ in $(0,l)$. Consider the following two problems
 \bess\left\{
 \!\begin{array}{ll}
 \ud u_t=d_1\ud u_{xx}-a\ud u+H(\ud v), &t>0,~x\in(0,l),\\[1mm]
\ud v_t=d_2\ud v_{xx}-b\ud v+G(\ud u), &t>0,~x\in(0,l),\\[1mm]
\ud u(t,0)=\ud v(t,0)=\ud u(t,l)=\ud v(t,l)=0, &t>0,\\[1mm]
\ud u(0,x)=\ep\phi(x), ~ \ud v(0,x)=\ep\psi(x), &0\le x\le l.
 \end{array}
\right.
 \eess
and
\bess\left\{
 \!\begin{array}{ll}
 \bar u_t=d_1\bar u_{xx}-a\bar u+H(\bar v), &t>0,~x\in(0,l),\\[1mm]
\bar v_t=d_2\bar v_{xx}-b\bar v+G(\bar u), &t>0,~x\in(0,l),\\[1mm]
\bar u(t,0)=\bar v(t,0)=\bar u(t,l)=\bar v(t,l)=0, &t>0,\\[1mm]
\bar u(0,x)=Mu_l(x), ~ \bar v(0,x)=Mv_l(x), &0\le x\le l.
 \end{array}
\right.
 \eess
 As above, it is easy to show that there exist $0<\ep\ll1$ and $M\gg1$ such that
 \bess\left\{\!\begin{array}{ll}
-d_1\ud u''\le-a\ud u+H(\ud v), \;\;&x\in(0,l),\\[1mm]
-d_2\ud v''\le-b\ud v+G(\ud u), &x\in(0,l),\\[1mm]
\ud u(0)=\ud v(0)=\ud u(l)=\ud v(l)=0,
 \end{array}\right.
\eess
and
\bess\left\{\!\begin{array}{ll}
-d_1\ol u''\ge-a\ol u+H(\ol v), \;\;&x\in(0,l),\\[1mm]
-d_2\ol v''\ge-b\ol v+G(\ol u), &x\in(0,l),\\[1mm]
\ol u(0)=\ol v(0)=\ol u(l)=\ol v(l)=0.
 \end{array}\right.
\eess
Owing to the comparison principle for parabolic systems, we have that $(\ud u,\ud v)$ is nondecreasing and $(\ol u,\ol v)$ is nonincreasing in $t\ge0$, as well as
 \[(\ud u(t,x),\ud v(t,x))\le (u(t,x),v(t,x))\le(\ol u(t,x),\ol v(t,x)),\;\;\;x\in[0,l],\;\;t\ge 0.\]

Then using the arguments of regularity, uniform estimates and compactness for parabolic equations (cf. \cite[Theorem 3.14]{WangPara}), one easily shows that both $(\ud u(t,x),\ud v(t,x))$ and $(\ol u(t,x),\ol v(t,x))$ converge to $(u_l(x),v_l(x))$ in $C^2([0,l])$ as $t\to\yy$. Hence $(u(t,x),v(t,x))\to(u_l(x),v_l(x))$ in $C^2([0,l])$ as $t\to\yy$. The conclusion (1) is obtained.

(2) The proof will be divided into two steps.

{\bf Step 1.} This step is devoted to the case with Dirichlet boundary condition $\mathbb{B}[w]=w$.  By arguing as in the proof of the uniqueness of (1), one directly obtains that $(u_l,v_l)$ is nondecreasing in $l>l_0$. Then from the strong maximum principle, it follows that $(u_l,v_l)$ is strictly increasing in $l>l_0$, i.e., $u_{l_1}(x)>u_{l_2}(x)$ and $v_{l_1}(x)>v_{l_2}(x)$ in $(0,l_2)$ for any $l_1>l_2>l_0$. Notice that $0<u_l<u^*$, $0<v_l<v^*$ in $(0,l)$. Thus $U(x):=\lim_{l\to\yy}u_l(x)$ and $V(x):=\lim_{l\to\yy}v_l(x)$ are well defined for all $x\ge0$. Moreover, $0<U\le u^*$ and $0<V\le v^*$ in $[0,\yy)$. By the regularity theory for elliptic equations, we know that $(u_l(x),v_l(x))\to(U(x),V(x))$ in $[C^2_{\rm loc}([0,\yy))]^2$, and thus $(U,V)$ satisfies \eqref{2.8}.

Then we prove the uniqueness.

{\it Claim 1.}  $(U,V)$ is strictly increasing in $x\ge0$ and $(U(x),V(x))\to(u^*,v^*)$ as $x\to\yy$.

By \cite[Theorem 1]{Troy}, $(u_l,v_l)$ is radially symmetric about $x=l/2$ and $u'(x)>0$, $v'(x)>0$ for $0<x<l/2$. Thus $(U,V)$ is nondecreasing in $x\ge0$. Let $(\dot{U},\dot{V})=(U',V')$. Then we have
\bes\left\{\!\begin{array}{ll}\label{2.11}
-d_1\dot U''=-a\dot{U}+H'(V)\dot{V}, \;\;&x\in(0,\yy),\\[1mm]
-d_2\dot V''=-b\dot{V}+G'(U)\dot{U}, &x\in(0,\yy),\\[1mm]
\dot U(0)>0, ~ \dot V(0)>0, ~ \dot{U}\ge0, ~ \dot{V}\ge0.
 \end{array}\right.
\ees
By the strong maximum principle for elliptic equations, we have $\dot{U}>0$ and $\dot{V}>0$, i.e., $U'>0$ and $V'>0$ in $[0,\yy)$. Then $U(\yy):=\lim_{x\to\yy}U(x)>0$ and $V(\yy):=\lim_{x\to\yy}V(x)>0$ are well defined. Notice that $(U',V')$ is uniformly continuous in $[0,\yy)$. By Barbalat's lemma or \cite[Lemma 2.3]{WZou}, we see that $(U'(x),V'(x))\to(0,0)$ as $x\to\yy$. By \eqref{2.11}, $(U'',V'')$ is also uniformly continuous in $[0,\yy)$. Using Barbalat's lemma again, we obtain $(U''(x),V''(x))\to(0,0)$ as $x\to\yy$. Letting $x\to\yy$ in \eqref{2.6} yields that $H(V(\yy))=aU(\yy)$ and $G(U(\yy))=bV(\yy)$, which, together with {\bf(H)}, implies that $U(\yy)=u^*$ and $V(\yy)=v^*$. Our claim is verified.

Let $(U_1,V_1)$ be an arbitrary bounded positive solution of \eqref{2.8}.

{\it Claim 2.} $(U_1(x),V_1(x))\to(u^*,v^*)$ as $x\to\yy$.

Using the similar arguments as in (1), we can show that $(U_1,V_1)\ge(u_l,v_l)$ in $[0,l]$ for all $l>l_0$. So $(U_1,V_1)\ge(U,V)$ in $[0,\yy)$. Note that $(U(x),V(x))\to(u^*,v^*)$ as $x\to\yy$. We have $\liminf_{x\to\yy}(U_1(x),V_1(x))\ge(u^*,v^*)$. It remains to show $U_{\rm sup}:=\sup_{x\in[0,\yy)}U_1\le u^*$ and $V_{\rm sup}:=\sup_{x\in[0,\yy)}V_1(x)\le v^*$. By way of contradiction, we may assume that $\sup_{x\in[0,\yy)}U_1(x)>u^*$. There are two cases to be considered.

{\it Case 1.} There exists some $x_0>0$ such that $U_1(x_0)=\sup_{x\in[0,\yy)}U_1(x)>u^*$.

{\it Case 2.} $U_1(x)<U_{\rm sup}$ for all $x\ge0$, and $\limsup_{x\to\yy}U_1(x)=U_{\rm sup}>u^*$. This case can be divided into two subcases, {\it subcase 1}: there exists a sequence of maximum pints $\{x_n\}\to\yy$ such that $U_1(x_n)\to U_{\rm sup}$, and {\it subcase 2}: there exists some $X_0>0$ such that $U_1(x)$ increases to $U_{\rm sup}$ for $x\ge X_0$.

Next we show both these two cases can derive some contradictions. For case 1, since $x_0$ is a maximum pint of $U_1$, we have $U''_1(x_0)\le0$. Thus $H(V_1(x_0))\ge aU_1(x_0)$. By {\bf (H)}, it is easy to see that $V_1(x_0)>v^*$. Thus $V_{\rm sup}>v^*$. If there is a $\hat x>0$ such that $V_1(\hat x)=V_{\rm sup}$, then $V''_1(\hat x)\le0$. So $G(U_1(\hat x))\ge bV_1(\hat x)$. In summary, we obtain
 \[H(V_1(x_0))\ge aU_1(x_0), ~ G(U_1(\hat x))\ge bV_1(\hat x), ~ U_1(x_0)\ge U_1(\hat{x}), ~ V_1(x_0)\le V_1(\hat{x}),\]
 by which we can deduce that $G(H(V_1(x_0))/a)\ge bV_1(x_0)$. In view of {\bf (H)}, there exists another positive root of \eqref{1.2}. This is a contradiction. If $V_1(x)<V_{\rm sup}$ for all $x\ge0$, then either there exists a sequence of maximum pints $\{\tilde x_n\}\to\yy$ such that $V_1(\tilde x_n)\to V_{\rm sup}$, or there exists some $\tilde X_0>0$ such that $V_1(x)$ increases to $V_{\rm sup}$ for $x\ge \tilde X_0$. For the former, we have $V''_1(\tilde{x}_n)\le0$. So $G(U_1(\tilde{x}_n))\ge bV_1(\tilde{x}_n)$. By passing a subsequence if necessary, we see $G(U^{\yy}_1)\ge bV_{\rm sup}$ for some $U^{\yy}_1>u^*$. In a word, we obtain
 \[H(V_1(x_0))\ge aU_1(x_0), ~ G(U^{\yy}_1)\ge bV_{\rm sup}, ~ U_1(x_0)\ge U^{\yy}_1, ~ V_1(x_0)\le V_{\rm sup},\]
 which also indicates $G(H(V_1(x_0))/a)\ge bV_1(x_0)$. Thus a similar contradiction can be derived.  For the latter, noticing that $V_1$ converges and $V''_1$ is bounded, we see $V'\to0$ as $x\to\yy$. Moreover, since $U''_1$ and $U_1$ are bounded, so does $U'_1$. Differentiating the equation of $V_1$ yields that $V'''_1$ is bounded too, which implies that $V''_1$ is uniformly continuous in $x\ge0$. This together with $V'_1\to0$ as $x\to\yy$ leads to $V''_1\to0$ as $x\to\yy$. Thus by passing a subsequence if necessary, we have $G(\tilde{U}^{\yy}_1)=bV_{\rm sup}$ for some $\tilde{U}^{\yy}_1>u^*$. Analogously, we can derive that there exists another positive root of \eqref{1.2}, which is a contradiction. Therefore, Case 1 always can produce some contradictions.

Now we handle Case 2. If there exists a sequence of maximum pints $\{x_n\}\to\yy$ such that $U_1(x_n)\to U_{\rm sup}>u^*$, then one easily obtains $V_{\rm sup}>v^*$. Moreover, by passing a subsequence if necessary, we have $H(V^{\yy}_1)\ge aU_{\rm sup}$ for some $V^{\yy}_1>v^*$. If $V_1$ achieves its supremum somewhere on $[0,\yy)$, then using the similar arguments as above, we can derive a contradiction. Then we deal with the case $V_1(x)<V_{\rm sup}$ for all $x\ge0$, and clearly $\limsup_{x\to\yy}V_1(x)=V_{\rm sup}>v^*$. Analogously, we can derive that there exists some $\hat U^{\yy}_1>u^*$ such that $G(\hat U^{\yy})\ge bV_{\rm sup}$, which combined with $H(V^{\yy}_1)\ge aU_{\rm sup}$, yields a contradiction.

 If there exists some $X_0>0$ such that $U_1(x)$ increases to $U_{\rm sup}$ for $x\ge X_0$, by following the above lines we also can obtain some contradictions. To sum up, our claim is proved.

From the above arguments, we can define
\[\hat k=\inf\{k\ge1: k(U,V)\ge(U_1,V_1) {\rm ~ in  ~ }[0,\yy)\}.\]
Clearly, $\hat k\ge1$ and $\hat k(U,V)\ge(U_1,V_1)$ in $[0,\yy)$. Assume $\hat{k}>1$. Denote $\hat{k}U-U_1$ by $\mathcal{U}$. Then
\bess
-d_1\mathcal{U}''+a\mathcal{U}=\hat kG(V)-G(V_1)>G(\hat{k}V)-G(V_1)\ge0, ~ ~ x\in(0,\yy).
\eess
Since $\mathcal{U}(0)=0$ and $\mathcal{U}\ge0$, from the strong maximum principle and Hopf lemma for elliptic equations, we have $\mathcal{U}'(0)>0$ and $\mathcal{U}>0$ in $(0,\yy)$. Thus there exists a small $\ep_0>0$ such that $\hat{k}>1+\ep_0$ and $\frac{\hat{k}}{1+\ep_0}U\ge U_1$ in $[0,1]$. Notice that $\mathcal{U}(x)\to(\hat{k}-1)u^*>0$. One easily finds some $\ep_1>0$ such that $\hat{k}>1+\ep_1$ and $\frac{\hat{k}}{1+\ep_1}U\ge U_1$ in $[1,\yy)$. Thus $\hat{k}>1+\min\{\ep_0,\ep_1\}$ and $\frac{\hat{k}}{1+\min\{\ep_0,\ep_1\}}U\ge U_1$ in $[0,\yy)$. Similarly, we can derive that there exists some $\ep_2>0$ such that $\hat{k}>1+\ep_2$ and $\frac{\hat{k}}{1+\ep_2}V\ge V_1$ in $[0,\yy)$. This contradicts the definition of $\hat{k}$. So $\hat{k}=1$ and $(U,V)\ge(U_1,V_1)$ in $[0,\yy)$. Therefore, the uniqueness is obtained.

{\bf Step 2.} We prove the conclusion (2) for  the Neumann boundary condition $\mathbb{B}[w]=w'$. Let $(u_l,v_l)$ be the unique positive solution of \eqref{2.6} with $\mathbb{B}[u]=u'$ and $\mathbb{B}[v]=v'$. Obviously, $(u_l,v_l)$ is decreasing with respect to $x\in[0,l]$. By (1), we can define $U_2:=\lim_{l\to\yy}u_l(x)$ and $V_2:=\lim_{l\to\yy}v_l(x)$ for $x\ge0$. Clearly, $0<U_2\le u^*$ and $0<V_2\le v^*$ in $[0,\yy)$ and $(U_2,V_2)$ is nonincreasing in $x\ge0$. By the regularity of elliptic equations, $(u_l,v_l)\to(U_2,V_2)$ in $[C^2_{\rm loc}([0,\yy))]^2$ and $(U_2,V_2)$ satisfies \eqref{2.8} with $u(0)=v(0)=0$ replaced by $u'(0)=v'(0)=0$. Let $(\tilde{u}_l,\tilde{v}_l)$ be the unique positive solution of \eqref{2.6} with $\mathbb{B}[u]=u$ and $\mathbb{B}[v]=v$. For all large $l>0$, as above we see $(u_l,v_l)\ge(\tilde{u}_l,\tilde{v}_l)$ for $x\in[0,l]$. Thus $(U_2,V_2)\ge(U,V)$ where $(U,V)$ is the unique bounded positive solution of \eqref{2.8} with $\mathbb{B}[u]=u$ and $\mathbb{B}[v]=v$, which together with $0<U_2\le u^*$ and $0<V_2\le v^*$ in $[0,\yy)$, yields that $(U_2(x),V_2(x))\to(u^*,v^*)$ as $x\to\yy$. Note that $(U_2,V_2)$ is nonincreasing in $x\ge0$. We immediately obtain $(U_2,V_2)=(u^*,v^*)$. The proof is finished.
\end{proof}

\begin{remark}\label{r2.1}As is seen from the above proof, \eqref{2.8} has a unique bounded positive solution $(U,V)$ which is strictly increasing in $[0,\yy)$ and connects $(0,0)$ and $(u^*,v^*)$. In a forthcoming paper, it will be proved that $(U,V)$ is the limiting profile of the solution of a semi-wave problem, and $(u^*-U(x),v^*-V(x))=e^{\alpha x}(p+o(1),q+o(1))$ for some $p,q>0$ and $\alpha<0$.
\end{remark}

Thanks to {\bf(H)} and $\mathcal{R}_0>1$, we have $G(H(x)/a)-bx<0$ for all $x>v^*$, which implies that we can choose sufficiently large $K_2>0$ such that
 \bes\label{2.11a}
 K_2>\max\{v^*,\|v_0\|_{C([0,h_0])}\}, ~ K_1:=\frac{H(K_2)}{a}> \max\{u^*,\|u_0\|_{C([0,h_0])}\}\; {\rm  ~ and ~ }\; G(K_1)< bK_2.
 \ees
In order to study the longtime behaviors of solution component $(u,v)$ of \eqref{1.6} when operator $\mathbb{B}[w]=w$, we need to consider the problem
\bes\left\{\!\begin{array}{ll}\label{2.12}
-d_1u''=-au+H(v), \;\;&x\in(0,l),\\[1mm]
-d_2v''=-bv+G(u), &x\in(0,l),\\[1mm]
u(0)=v(0)=0, ~ u(l)=K_1, ~ v(l)=K_2,
 \end{array}\right.
\ees
where $(K_1,K_2)$ satisfies \eqref{2.11a}.

We would like to mention that the well-known moving plane method, which has a long history and been attracting many people up to now, will be used later to show the monotonicity.
This clever method usually can be used to prove the symmetry, monotonicity and a prior estimate for elliptic equations and nonlocal equations. Please see, for example, \cite{Duorder,ChenLi} and references therein for single equation, as well as \cite{Troy,Dja} for cooperative elliptic systems on bounded domain.

\begin{lem}\label{l2.4}Assume that $\mathcal{R}_0>1$. Then the following statements hold.
\begin{enumerate}[$(1)$]
\item Problem \eqref{2.12} has a unique positive solution $(u_l,v_l)$ if $\lambda(l)<0$, i.e., $l>l_0$, where $\lambda(l)$ is the principal eigenvalue of \eqref{2.3} with operator $\mathbb{B}[w]=w$.  Moreover, $0<u_l<K_1$ and $0<v_l<K_2$, $u'_l>0$ and $v'_l>0$ in $(0,l)$, and $(u_l,v_l)$ is globally asymptotically stable.
\item The unique positive solution $(u_l,v_l)$ is strictly decreasing in $l>l_0$. Besides, $(u_l,v_l)\to(U,V)$ in $[C^2_{\rm loc}([0,\yy))]^2$ as $l\to\yy$ where $(U,V)$ is the unique bounded positive solution of \eqref{2.8}.
    \end{enumerate}
\end{lem}
\begin{proof}(1) Since $\lambda(l)<0$, owing to Lemma \ref{l2.3}, problem \eqref{2.6} with $\mathbb{B}[u]=\mathbb{B}[v]=0$ has a unique positive solution $(\tilde{u}_l,\tilde{v}_l)$. Thus it is easy to see that $(\tilde{u}_l,\tilde{v}_l)$ and $(K_1,K_2)$ are the ordered upper and lower solutions for \eqref{2.12}. By the upper and lower solutions method for elliptic systems, \eqref{2.12} has at least one positive solution $(u_l,v_l)$ with $(\tilde{u}_l,\tilde{v}_l)\le(u_l,v_l)\le(K_1,K_2)$ in $[0,l]$. The uniqueness can be derived by utilizing the similar arguments as in Lemma \ref{l2.3}. By the strong maximum principle for elliptic equations, we easily show $0<u_l<K_1$ and $0<v_l<K_2$ in $(0,l)$. The stability also can be proved by using the methods as in Lemma \ref{l2.3}.

Thus to complete the proof of (1), it remains to show $u'_l>0$ and $v'_l>0$ in $(0,l)$. It will be done by using the well-known moving plane method. For convenience, we drop $l$ in $(u_l,v_l)$, and rewrite $(u_l,v_l)$ as $(u,v)$. Choose $\lambda\in(0,l/2)$ and denote $(u(2\lambda-x),v(2\lambda-x))$ by $(u^{\lambda}(x),v^{\lambda}(x))$. Clearly, $(u^{\lambda},v^{\lambda})$ satisfies the first two equations of \eqref{2.12} in $(0,\lambda)$. Define
\[w(x,\lambda)=u^{\lambda}(x)-u(x), ~ ~ z(x,\lambda)=v^{\lambda}(x)-v(x).\]
Then we have
\bes\left\{\!\begin{array}{ll}\label{2.13}
-d_1w''=-aw+\xi z, \;\;&x\in(0,\lambda),\\[1mm]
-d_2z''=-bz+\eta w, &x\in(0,\lambda),\\[1mm]
w(0)>0, ~ z(0)>0, ~ w(\lambda)=z(\lambda)=0,
 \end{array}\right.
\ees
where $(\xi,\eta)\ge(0,0)$ and $\xi,\eta\in L^{\yy}((0,\lambda))$. By \cite[Proposition 1.1]{Dja}, there exists a $0<\delta\ll1$ depending only on $d_1,d_2$ and $l$ such that $w\ge0$ and $z\ge0$ in $(0,\lambda)$ if $\lambda<\delta$. Then by the strong maximum principle for elliptic equations, we have $w>0$ and $z>0$ in $(0,\lambda)$ with  $\lambda\in(0,\delta)$. So we can define
\bess
\varrho=\sup\{\lambda\in(0,l/2):w>0,~ z>0 {\rm ~ in ~ }(0,\lambda)\}.
\eess
Thanks to the strong maximum principle, one sees $w(x,\varrho)>0$ and $z(x,\varrho)>0$ in $(0,\varrho)$. We next show $\varrho=l/2$. Argue by the contradiction that $\varrho<l/2$. Clearly, $w(x,\varrho)>0$ and $z(x,\varrho)>0$ on $[\delta/8,\varrho-\delta/8]$. Due to the continuity, there exists a small $\ep_0\in(0,3\delta/4)$ such that \[w(x,\varrho+\ep)>0, ~ ~z(x,\varrho+\ep)>0 ~ {\rm on ~ }[\delta/8,\varrho-\delta/8], ~ ~ \forall\ep\in[0,\ep_0].\]
Recall that $(w(x,\varrho+\ep),z(x,\varrho+\ep))$ satisfies the first two equations of \eqref{2.13}. Moreover,
\[w(x,\varrho+\ep)>0, ~ z(x,\varrho+\ep)>0 ~ {\rm at} ~ \{0,\delta/8,\varrho-\delta/8\} {\rm  ~ and  ~ }w(\varrho+\ep,\varrho+\ep)=z(\varrho+\ep,\varrho+\ep)=0.\]
 Thus owing to \cite[Proposition 1.1]{Dja} and the strong maximum principle, we know $w(x,\varrho+\ep)>0$ and $z(x,\varrho+\ep)>0$ on $[0,\delta/8]$ and $[\varrho-\delta/8,\varrho+\ep)$. Therefore, we derive that
 \[w(x,\varrho+\ep)>0, ~ z(x,\varrho+\ep)>0 ~ {\rm in ~ }(0,\varrho+\ep)~ {\rm  for ~ } \ep\in(0,\ep_0]\]
  which contradicts the definition of $\varrho$. So $\varrho=l/2$. Then for any $x_0\in(0,l/2)$, we immediately have
  \[u(x_0)<u(x), ~ v(x_0)<v(x) ~ ~\forall x_0<x<l-x_0,\]
  which implies that $u'(x)\ge0$ and $v'(x)\ge0$ on $[0,l/2]$. Using the Hopf lemma, we have $u'(0)>0$ and $v'(0)>0$. Then differentiating the equations of $u$ and $v$ yields
\bess\left\{\!\begin{array}{ll}
-d_1u'''=-au'+H'(v)v', \;\;&x\in(0,l/2),\\[1mm]
-d_2v'''=-bv'+G'(u)u', &x\in(0,l/2),\\[1mm]
u'(0)>0, ~ v'(0)>0, ~ u'(l/2)\ge0, ~ v'(l/2)\ge0.
 \end{array}\right.
\eess
Utilizing the strong maximum principle for the equations of $u'$ and $v'$, respectively, we get $u'>0$ and $v'>0$ in $(0,l/2)$.

Now we show $u'>0$ and $v'>0$ in $(l/2,l)$. Denote $(\tilde{u}(x),\tilde{v}(x))$ by $(u(l-x),v(l-x))$. Obviously, $(\tilde{u},\tilde{v})$ still satisfies the first two equations of \eqref{2.12}, but $(\tilde{u}(0),\tilde{v}(0))=(K_1,K_2)$ and $(\tilde{u}(l),\tilde{v}(l))=(0,0)$. Denote
\[\tilde{w}(x,\lambda)=\tilde{u}(x)-\tilde{u}(2\lambda-x),~ ~ \tilde{z}(x,\lambda)=\tilde{v}(x)-\tilde{v}(2\lambda-x).\]
 It is easy to see that $(\tilde{w}(x,\lambda),\tilde{z}(x,\lambda))$ satisfies \eqref{2.13}. Then we can argue as above to derive that $\tilde{u}'<0$ and $\tilde{v}'<0$ in $(0,l/2)$ which implies that $u'>0$ and $v'>0$ in $(l/2,l)$. Since $u'(l/2)>0$ and $v'(l/2)>0$ are obvious, we complete the proof of (1).

(2) Notice that $0<u_l<K_1$ and $0<v_l<K_2$ in $(0,l)$. Similar to the proof of Lemma \ref{l2.3}, it is not hard to prove that $(u_l,v_l)$ is strictly decreasing in $l>l_0$. Thus $\tilde U(x):=\lim_{l\to\yy}u_l$ and $\tilde{V}(x):=\lim_{l\to\yy}v_l(x)$ are well defined, nondecreasing in $[0,\yy)$ and $(\tilde{U},\tilde{V})\ge(U,V)$ in $[0,\yy)$. By the regularity for elliptic equations, we know $(u_l,v_l)\to(\tilde{U},\tilde{V})$ in $[C^2_{\rm loc}([0,\yy))]^2$ as $l\to\yy$, which implies $(\tilde{U},\tilde{V})$ solves \eqref{2.8}. By the uniqueness of bounded positive solution of \eqref{2.8}, we have $(\tilde{U},\tilde{V})=(U,V)$. The proof is complete.
\end{proof}

\begin{remark}\label{r2.2}From the above lemma, it follows that \eqref{2.12} has a unique positive solution if $\lambda(l)<0$, i.e., $l>l_0$, where $\lambda(l)$ is the principal eigenvalue of \eqref{2.3} with operator $\mathbb{B}[w]=w$.

Indeed, we can prove that if $\lambda(l)<0$, where $\lambda(l)$ is the principal eigenvalue of \eqref{2.3} with operator $\mathbb{B}[w]=w'$, then \eqref{2.12} also has a unique positive solution. Let $(u_l,v_l)$ be the positive solution of \eqref{2.6} with $\mathbb{B}[w]=w'$. One can easily verify that $(u_l(l-x),v_l(l-x))$ and $(K_1,K_2)$ are the ordered upper and lower solution of \eqref{2.12}. Then arguing as above, we obtain the desired result.
\end{remark}
\section{Dynamics of \eqref{1.6}}

With the aid of the results in the former section, we now study the dynamics of \eqref{1.6}, involving the well-posedeness, spreading-vanishing dichotomy, as well as criteria for spreading and vanishing. The well-posedeness of \eqref{1.6}, i.e., Theorem \ref{t1.1},  can be proved by virtue of analogous arguments to those in \cite{WD, Wdcdsb2021}. So we omit the details.

Then we discuss the longtime behaviors of solution component $(u,v)$, in which the principal eigenvalue $\lambda(l)$ of \eqref{2.3} plays a crucial role. In what follows, we always assume that $(u,v,h)$ is the unique solution of \eqref{1.6}.

\begin{lem}\label{l3.1}If $h_{\yy}<\yy$, then
\bes\left\{\!\begin{array}{ll}\label{3.1}
\dd\lim_{t\to\yy}\|u(t,x)+v(t,x)\|_{C([0,h(t)])}=0, ~ ~ \lambda(h_{\yy})\ge0,\\
\dd\lim_{t\to\yy}{\rm e}^{kt}\|u(t,x)+v(t,x)\|_{C([0,h(t)])}=0, ~ ~ \forall k\in(0,\lambda(h_{\yy})), ~ ~ {\rm if ~ }\lambda(h_{\yy})>0.
 \end{array}\right.
\ees
\end{lem}
\begin{proof}The proof will be completed by two steps.

{\bf Step 1.} This step involves the proof of the assertions in the first line of \eqref{3.1}. By using the similar arguments as in the proofs of \cite[Theorem 2.3]{LLW}, we can obtain the following estimates
\[\|u(t,\cdot)+v(t,\cdot)\|_{C^1([0,h(t)])}+\|h'\|_{C([1,\yy))}\le C ~ ~ \forall t\ge1,\]
where $C$ only depends on the parameters in \eqref{1.6} and the initial function $(u_0,v_0)$. The above estimate implies that $h'$ is uniformly continuous in $[0,\yy)$, which combined with $h_{\yy}<\yy$ yields $h'(t)\to0$ as $t\to\yy$. Since the positivity of $v$, the solution component $(u,h)$ satisfies
\bess\left\{\!\begin{array}{ll}
u_t-d_1u_{xx}\ge-au, &t>0,~x\in(0,h(t)),\\[1mm]
\mathbb{B}[u](t,0)=u(t,h(t))=0, &t>0,\\[1mm]
h'(t)\ge-\mu_1 u_x(t,h(t)), & t>0,\\[1mm]
h(0)=h_0, ~ u(0,x)=u_{0}(x)\ge,\not\equiv0,&0\le x\le h_0.
 \end{array}\right.
 \eess
 Then in view of \cite[Lemma 3.3]{LLW}, we immediately have $\lim_{t\to\yy}\|u(t,x)\|_{C([0,h(t)])}=0$. Analogously, we can derive $\lim_{t\to\yy}\|v(t,x)\|_{C([0,h(t)])}=0$. Thus to finish this step, it remains to prove $\lambda(h_{\yy})\ge0$. Assume on the contrary that $\lambda(h_{\yy})<0$. By continuity, there exists a $T>0$ such that $\lambda(h(T))<0$. Consider the following auxiliary problem
 \bess\left\{
 \!\begin{array}{ll}
 \ud u_t=d_1\ud u_{xx}-a\ud u+H(\ud v), &t>T,~x\in(0,h(T)),\\[1mm]
\ud v_t=d_2\ud v_{xx}-b\ud v+G(\ud u), &t>T,~x\in(0,h(T)),\\[1mm]
\mathbb{B}[\ud u](t,0)=\mathbb{B}[\ud v](t,0)=\ud u(t,h(T))=\ud v(t,h(T))=0, &t>T,\\[1mm]
\ud u(T,x)=\frac{1}{2}u(T,x), ~ \ud v(0,x)=\frac{1}{2}v(T,x), &0\le x\le h(T).
 \end{array}
\right.
 \eess
 By the comparison principle for parabolic equations of cooperative systems, we have $(u,v)\ge (\ud u,\ud v)$ for $t\ge T$ and $x\in[0,h(T)]$. Notice that $\lambda(h(T))<0$. From Lemma \ref{l2.3}, it follows that $(\ud u,\ud v)\to(u_{h(T)},v_{h(T)})$ in $[C^2([0,h(T)])]^2$ as $t\to\yy$, where $(u_{h(T)},v_{h(T)})$ is the unique positive solution of \eqref{2.6}. Thus $\liminf_{t\to\yy}(u,v)\ge (u_{h(T)},v_{h(T)})$ uniformly in $[0,h(T)]$, which contradicts $\lim_{t\to\yy}\|u(t,x)+v(t,x)\|_{C([0,h(t)])}=0$. So $\lambda(h_{\yy})\ge0$. This step is finished.

 {\bf Step 2.} In this step, we prove the assertion in the second line of \eqref{3.1}. Recall $\lambda(h_{\yy})>0$. Let $(\phi,\psi)$ be the positive eigenfunction of $\lambda(h_{\yy})$. Define $(\bar{u},\bar{v})=(M{\rm e}^{-kt}\phi,M{\rm e}^{-kt}\psi)$ where $k\in(0,\lambda(h_{\yy}))$ and $M$ is to be determined later. Straightforward computations yields that for $t>0$ and $x\in(0,h_{\yy})$,
 \bess
 &&\bar{u}_t-d_1\bar{u}_{xx}+a\bar{u}-H(\bar{v})\\
 &&=M{\rm e}^{-kt}\kk(-k\phi+\lambda(h_{\yy})+H'(0)\psi
 -\frac{H(\bar{v})}{M{\rm e}^{-kt}}\rr)\ge M{\rm e}^{-kt}\big[-k\phi+\lambda(h_{\yy})\big]\ge0.
 \eess
 Similarly, we have
 \[\bar{v}_t\ge d_2\bar{v}_{xx}-b\bar{v}+G(\bar{u}) ~ ~ {\rm for ~ }(t,x)\in(0,\yy)\times(0,h_{\yy}).\]
Moreover, it is easy to verify that there is a $M\gg1$ such that $(\bar{u}(0,x),\bar{v}(0,x))=(M\phi(x),M\psi(x))\ge(u_0(x),v_0(x))$ for $x\in[0,h_0]$. Clearly, $\mathbb{B}[\bar{u}](0,0)=\mathbb{B}[\bar{v}](0,0)=\bar{u}(0,h_{\yy})=\bar{v}(0,h_{\yy})=0$. By a comparison argument, we have $(u,v)\le (\bar{u},\bar{v})$ for $t\ge0$ and $x\in[0,h(t)]$. The desired result follows from the definition of $(\bar{u},\bar{v})$, which completes the step 2. The proof is ended.
\end{proof}

The above lemma shows that if vanishing occurs, $(u,v)\to(0,0)$ uniformly in $[0,h(t)]$ as $t\to\yy$ which implies that the epidemic will disappear in the long run. Next we consider the case spreading.

\begin{lem}\label{l3.2}If $h_{\yy}=\yy$, then
\bess\left\{\!\begin{array}{ll}
\dd\lim_{t\to\yy}(u,v)=(U,V) ~ {\rm in ~ }C_{\rm loc}([0,\yy)) {\rm ~ when ~ operator ~ }\mathbb{B}[w]=w,\\
\dd\lim_{t\to\yy}(u,v)=(u^*,v^*) ~ {\rm in ~ }C_{\rm loc}([0,\yy)) {\rm ~ when ~ operator ~ }\mathbb{B}[w]=w',
 \end{array}\right.
\eess
where $(u^*,v^*)$ is given by \eqref{1.2} and $(U,V)$ is the unique bounded positive solution of \eqref{2.8}.
\end{lem}
\begin{proof} Necessarily, $\mathcal{R}_0>1$ (see Lemma \ref{l3.3}). The proof is divided into two steps.

{\bf Step 1.} In this step, we prove the statement for the Dirichlet boundary condition $\mathbb{B}[w]=w$.  For any large $l>l_0$, there exists a $T>0$ such that $h(T)>l$.  Consider the following two auxiliary problems
\bess\left\{\!\begin{array}{ll}
 \ud u_t=d_1\ud u_{xx}-a\ud u+H(\ud v), &t>T,~x\in(0,l),\\[1mm]
\ud v_t=d_2\ud v_{xx}-b\ud v+G(\ud u), &t>T,~x\in(0,l),\\[1mm]
\ud u(t,0)=\ud v(t,0)=\ud u(t,l)=\ud v(t,l)=0, &t>T,\\[1mm]
\ud u(T,x)=\ud u_0(x), ~ \ud v(T,x)=\ud v_0(x), &0\le x\le l,\\
 \end{array}
\right.
 \eess
 and
 \bess\left\{\!\begin{array}{ll}
 \ol u_t=d_1\ol u_{xx}-a\ol u+H(\ol v), &t>T,~x\in(0,l),\\[1mm]
\ol v_t=d_2\ol v_{xx}-b\ol v+G(\ol u), &t>T,~x\in(0,l),\\[1mm]
\ol u(t,0)=\ol v(t,0)=0, ~ \ol u(t,l)=K_1, ~ \ol v(t,l)=K_2, &t>T,\\[1mm]
\ol u(T,x)=\ol u_0(x), ~ \ol v(T,x)=\ol v_0(x), &0\le x\le l,\\
 \end{array}
\right.
 \eess
where $(\ud u_0,\ud v_0)\le (u(T,x),v(T,x))\le (\ol u_0,\ol v_0)$ in $[0,l]$, $(K_1,K_2)$ satisfies \eqref{2.11} and $(K_1,K_2)\ge(C_1,C_2)$ which is given in Theorem \ref{t1.1}. Moreover, the compatibility conditions hold for $(\ud u_0,\ud v_0)$ and $(\ol u_0,\ol v_0)$. By the comparison principle for parabolic systems, we have
 \[(\ud u(t,x),\ud v(t,x))\le(u(t,x),v(t,x))\le(\ol u(t,x),\ol v(t,x)) {\rm ~ ~ for }~ t\ge T ~ {\rm and ~ }x\in[0,l].\]
  Using Lemmas \ref{l2.3} and \ref{l2.4} yields  $(\ud u(t,x),\ud v(t,x))\to(u_l(x),v_l(x))$ and $(\ol u(t,x),\ol v(t,x))\to(\bar u_l(x),\bar v_l(x))$ in $C([0,l])$ as $t\to\yy$, where $(u_l,v_l)$ and $(\bar{u}_l,\bar{v}_l)$ are the unique positive solution of \eqref{2.6} and \eqref{2.12}, respectively. Thus we see
  \[(u_l(x),v_l(x))\le\liminf_{t\to\yy}(u(t,x),v(t,x))\le\limsup_{t\to\yy}(u(t,x),v(t,x))\le(\bar{u}_l(x),\bar{v}_l(x)) {\rm ~ in ~ }C([0,l]).\]
  Since both $(u_l(x),v_l(x))$ and $(\bar{u}_l(x),\bar{v}_l(x))$ converge to $(U,V)$ in $C_{\rm loc}([0,\yy))$ as $l\to\yy$, where $(U,V)$ is the unique bounded positive solution of \eqref{2.8}, we can derive the result as wanted. The step 1 is finished.

{\bf Step 2.} We now deal with the case $\mathbb{B}[w]=w'$. Consider the ODE system
  \[\ol u_t=-a\ol u+H(\ol v), ~ \ol v_t=-b\ol v+G(\ol u); ~ ~ \ol u(0)=\|u_0\|_{C([0,h_0])}, ~ \ol v(0)=\|v_0\|_{C([0,h_0])}.\]
  In view of $\mathcal{R}_0>1$, we see that $(\ol u,\ol v)\to(u^*,v^*)$ as $t\to\yy$. Using a comparison argument yields that $\limsup_{t\to\yy}(u(t,x),v(t,x))\le(u^*,v^*)$ uniformly in $[0,\yy)$.  For any $l>l_0$, choose $T$ large enough such that $h(T)>l$.
 Let (\ud u, \ud v) be the solution of problem
 \bess\left\{\!\begin{array}{ll}
 \ud u_t=d_1\ud u_{xx}-a\ud u+H(\ud v), &t>T,~x\in(0,l),\\[1mm]
\ud v_t=d_2\ud v_{xx}-b\ud v+G(\ud u), &t>T,~x\in(0,l),\\[1mm]
\ud u_x(t,0)=\ud v_x(t,0)=\ud u(t,l)=\ud v(t,l)=0, &t>T,\\[1mm]
\ud u(T,x)=\ud u_0(x), ~ \ud v(T,x)=\ud v_0(x), &0\le x\le l,\\
 \end{array}
\right.
 \eess
 where $(\ud u_0,\ud v_0)\le(u(T,x),v(T,x))$ in $[0,l]$ and the compatibility condition holds for $(\ud u_0,\ud v_0)$. From a comparison method and Lemma \ref{l2.3}, we immediately obtain $\liminf_{t\to\yy}(u(t,x),v(t,x))\ge(u^*,v^*)$ locally uniformly in $[0,\yy)$, which together with our early result, completes the step 2. Therefore, the proof is ended.
\end{proof}

Combining Lemmas \ref{l3.1} with \ref{l3.2}, we immediately obtain a spreading-vanishing dichotomy, namely, Theorem \ref{t1.2}.
In what follows, we investigate the criteria for spreading and vanishing.

\begin{lem}\label{l3.3}If $\mathcal{R}_0\le1$, then vanishing happens. Moreover,
\[h_{\yy}\le h_0+\frac{\max\{\mu_1,\mu_2\}}{\min\{d_1,\frac{H'(0)d_2}{b}\}}\int_{0}^{h_0}
\left(u_0(x)+\frac{H'(0)}{b}v_0(x)\right)\dx.\]
\end{lem}
\begin{proof}
By the Hopf lemma for parabolic equations, we see $u_x(t,h(t))<0$, $v_x(t,h(t))<0$, $u_x(t,0)>0$ and $v_x(t,0)>0$ if $\mathbb{B}[w]=w$. By virtue of {\bf (H)} and $\mathcal{R}_0\le1$, we have
  \bess
\frac{\rm d}{{\rm d}t}\int_{0}^{h(t)}\left(u+\frac{H'(0)}{b}v\right)\dx&=& \int_{0}^{h(t)}\left(u_t+\frac{H'(0)}{b}v_t\right)\dx\\[2mm]
&=&d_1u_x(t,h(t))+d_2v_x(t,h(t))-d_1u_x(t,0)-d_2v_x(t,0)\\[2mm]
&&+\int_{0}^{h(t)}\kk(-au+H(v)-H'(0)v+\frac{H'(0)}{b}G(u)\rr)\dx\\
&\le&d_1u_x(t,h(t))+d_2v_x(t,h(t))\le\frac{-\min\{d_1,\frac{H'(0)d_2}{b}\}}{\max\{\mu_1,\mu_2\}}h'.
\eess
Integrating the above inequality from $0$ to $t$ yields
\[h(t)\le h_0+\frac{\max\{\mu_1,\mu_2\}}{\min\{d_1,\frac{H'(0)d_2}{b}\}}
\int_{0}^{h_0}\left(u_0(x)+\frac{H'(0)}{b}v_0(x)\right)\dx,\]
which completes the proof.
\end{proof}

Next we focus on the case $\mathcal{R}_0>1$ which is assumed to be valid in the remainder of this section. According to Lemma \ref{l2.2}, there exists a unique critical length $l_0$, defined by
\bes\left\{\!\begin{array}{ll}\label{3.3}
l_0=\dd\pi\sqrt{\frac{ad_2+bd_1+\sqrt{(ad_2-bd_1)^2
 +4d_1d_2H'(0)G'(0)}}{2(H'(0)G'(0)-ab)}}\;\;{\rm ~ when ~ operator ~ }\mathbb{B}[w]=w,\\[4mm]
 l_0=\dd\frac{\pi}{2}\sqrt{\frac{ad_2+bd_1+\sqrt{(ad_2-bd_1)^2
 +4d_1d_2H'(0)G'(0)}}{2(H'(0)G'(0)-ab)}}\;\; {\rm ~ when ~ operator ~  }\mathbb{B}[w]=w',
 \end{array}\right.
\ees
such that $\lambda(l_0)=0$ and $\lambda(l)(l-l_0)<0$ for $l\neq l_0$. Thus taking advantaging of Lemma \ref{l3.1}, we directly have the following result.

\begin{lem}\label{l3.4}If vanishing occurs, then $h_{\yy}\le l_0$ where $l_0$ is defined in \eqref{3.3}. This implies that if $h_0\ge l_0$, then spreading happens.
\end{lem}

The above result shows that if initial habitat is beyond the critical size $l_0$, then spreading must occur. Our next conclusion suggests that if $h_0<l_0$ but $\mu_1$ or $\mu_2$ is large enough, then spreading also will happen.

\begin{lem}\label{l3.5} Assume that $h_0<l_0$. There exists a $\ol \mu_1>0$ ($\ol \mu_2>0$), which is independent of $\mu_2$ ($\mu_1$) such that spreading happens if $\mu_1\ge\ol\mu_1$ ($\mu_2\ge\ol\mu_2$).
\end{lem}
\begin{proof}We only prove the assertion about $\mu_1$ since the other case can be handled analogously. By the maximum principle and Hopf lemma for parabolic equations, we have that $u\ge0$, $v\ge0$ and $u_x(t,h(t))<0$ and $v_x(t,h(t))<0$ for $t>0$ and $x\in[0,h(t)]$. Thus $(u,h)$ satisfies
\bess\left\{\!\begin{array}{ll}
u_t\ge d_1u_{xx}-au, &t>0,~x\in(0,h(t)),\\[1mm]
\mathbb{B}[u](t,0)=u(t,h(t))=0, &t>0,\\[1mm]
h'(t)>-\mu_1 u_x(t,h(t)), & t>0,\\[1mm]
h(0)=h_0, ~ u(0,x)=u_{0}(x),&0\le x\le h_0.
 \end{array}\right.
 \eess
 From \cite[Lemma 3.2]{WZrwa15}, there exists a $\ol\mu_1$ depending only on $(d_1, d,d_2,a, H'(0),G'(0),h_0,u_0)$ such that $\lim_{t\to\yy}h(t)>l_0$ if $\mu_1\ge\ol\mu_1$, which indicates that spreading happens. The proof is ended.
\end{proof}

The above result show that if $h_0<l_0$  and one of the expanding coefficients $\mu_1$ and $\mu_2$ is large enough, then spreading occurs. One naturally thinks what will happen if both $\mu_1$ and $\mu_2$ are small. The next result gives an answer.

\begin{lem}\label{l3.6}Suppose that $h_0<l_0$. Then there exists a $\bar{\mu}>0$ such that vanishing happens if $\mu_1+\mu_2<\bar{\mu}$.
\end{lem}
\begin{proof}
Due to $h_0<l_0$, by Lemma \ref{l2.2}, we see $\lambda(h_0)>0$. Denote by $(\phi,\psi)$ the positive eigenfunction of $\lambda(h_0)$. By Lemma \ref{l2.2} again, we have $\phi=p\omega$ and $\psi=\omega$ where
\[p=\frac{H'(0)}{d_1\nu_1(h_0)+a-\lambda(h_0)}>0\]
and $(\nu_1(h_0),\omega)$ is the principal eigenpair of \eqref{2.2}. Define
\bess
\bar{h}=h_0(1+\delta-\delta {\rm e}^{-\delta t}), ~ \bar{u}=M{\rm e}^{-\delta t}\phi\kk(\frac{xh_0}{\bar{h}(t)}\rr), ~ \bar{u}=M{\rm e}^{-\delta t}\psi\kk(\frac{xh_0}{\bar{h}(t)}\rr),
\eess
where $\delta$ and $M$ are to be chosen later.

We now show that $(\bar{u},\bar{v},\bar{h})$ is an upper solution to the unique solution $(u,v,h)$ of \eqref{1.6}, i.e.,
\bes\left\{\!\begin{array}{ll}\label{3.5}
\bar u_t\ge d_1\bar u_{xx}-a\bar u+H(\bar v), &t>0,~x\in(0,\bar h(t)),\\[1mm]
\bar v_t\ge d_2\bar v_{xx}-b\bar v+G(\bar u), &t>0,~x\in(0,\bar h(t)),\\[1mm]
\mathbb{B}[\bar u](t,0)\ge0, ~ \mathbb{B}[\bar v](t,0)\ge0, ~ \bar u(t,\bar h(t))\ge0, ~ \bar v(t,\bar h(t))\ge0, &t>0,\\[1mm]
\bar h'(t)\ge-\mu_1 \bar u_x(t,\bar h(t))-\mu_2\bar v_x(t,\bar h(t)), & t>0,\\[1mm]
\bar h(0)\ge h_0, ~ \bar u(0,x)\ge u_{0}(x), ~ \bar v(0,x)\ge v_0(x),&0\le x\le h_0.
 \end{array}\right.
 \ees

 Firstly, owing to \eqref{2.3} we can find a suitably large $M>0$ such that $\bar{u}(0,x)=M\phi(x)\ge u_0(x)$ and $\bar{v}(0,x)=M\psi(x)\ge v_0(x)$ in $[0,h_0]$. Moreover, it is easy to see that there exists a constant $K>0$ such that $x\phi'\le K\phi$  and $x\psi'\le K\phi$ in $[0,h_0]$.
For convenience, we denote $(xh_0)/\bar{h}$ by $y$.
Direct computations yield that for $t>0$ and $x\in(0,\bar{h}(t))$,
\bess
&&\bar{u}_t-d_1\bar{u}_{xx}+a\bar{u}-H(\bar{v})\\
&&=-\delta\bar{u}-M{\rm e}^{-\delta t}\phi'(y)\frac{yh_0\delta^2{\rm e}^{-\delta t}}{\bar{h}}-d_1M{\rm e}^{-\delta t}\phi''(y)(\frac{h_0}{\bar{h}})^2+a\bar{u}-H(\bar{v})\\[1mm]
&&\ge\bar{u}\kk[-\delta-\frac{\phi'(y)yh_0\delta^2{\rm e}^{-\delta t}}{\phi(y)\bar{h}}+a-\frac{H'(0)}{p}+(\frac{h_0}{\bar{h}})^2\kk(\lambda(h_0)
-a+\frac{H'(0)}{p}\rr)\rr]\\[1mm]
&&\ge\bar{u}\kk[-\delta-K\delta^2+\frac{\lambda(h_0)}{(1+\delta)^2}+
\frac{H'(0)}{p}\kk(\frac{1}{(1+\delta)^2}-1\rr)\rr]\ge0
\eess
provided that $\delta$ is sufficiently small. Similarly, we can show that for $t>0$ and $x\in(0,\bar{h}(t))$, $\bar{v}_t\ge d_2\bar{v}_{xx}-b\bar{v}+G(\bar{u})$.
Thus to prove \eqref{3.5}, it remains to verify the inequality in the fourth line of \eqref{3.5} since those in the third line are obvious.
Straightforward calculations leads to
\bess
&&-\mu_1\bar{u}_x(t,\bar{h})-\mu_2\bar{v}_x(t,\bar{h})\\
&&=-\mu_1M{\rm e}^{-\delta t}\phi'(h_0)\frac{h_0}{\bar{h}}-\mu_2M{\rm e}^{-\delta t}\psi'(h_0)\frac{h_0}{\bar{h}}\\
&&\le-\mu_1M{\rm e}^{-\delta t}\phi'(h_0)-\mu_2M{\rm e}^{-\delta t}\psi'(h_0)\\
&&\le(\mu_1+\mu_2)M{\rm e}^{-\delta t}(-\phi'(0)-\psi'(h_0))\le\delta^2h_0{\rm e}^{-\delta t}=\bar{h}',
\eess
if $\mu_1+\mu_2\le\frac{\delta^2h_0}{-M(\phi'(0)+\psi'(h_0)}$. So, in view of the comparison principle (Lemma \ref{l2.1}), we see that
\bess (\bar{u}(t,x),\bar{v}(t,x),\bar{h}(t))\ge(u(t,x),v(t,x),h(t)) ~ ~ \forall (t,x)\in[0,\yy)\times[0,h(t)],
\eess
which implies $\lim_{t\to\yy}h(t)\le \lim_{t\to\yy}\bar{h}(t)=h_0(1+\delta)<\yy$. The proof is finished.
\end{proof}

\begin{remark}\label{r3.1}From the above proof, it follows that vanishing happens, if $u_0$ and $v_0$ are small enough. This will be used later to derive a critical value for the parameterized initial function.
\end{remark}
According to Lemmas \ref{l3.4} and \ref{l3.5}, spreading happens if $\mu_1+\mu_2$ is large, while vanishing occurs if $\mu_1+\mu_2$ is small. One naturally thinks whether there exists a critical value for $\mu_1+\mu_2$ such that spreading happens if and only if $\mu_1+\mu_2$ is beyond this value. Indeed, such value does not exist since the unique solution $(u,v,h)$ usually is not monotone about $\mu_1+\mu_2$. But if we let $\mu_2=Q(\mu_1)$ with $Q\in C([0,\yy))$, $Q(0)=0$ and strictly increasing to $\yy$, we can find a critical value as wanted.

\begin{lem}\label{l3.7}Suppose that $h_0<l_0$ and $\mu_2=Q(\mu_1)$ with $Q$ defined as above. Then there exists a unique $\mu^*_1>0$ such that spreading happens if and only if $\mu_1>\mu^*_1$.
\end{lem}
\begin{proof}
By Lemma \ref{l3.4}, we know that spreading happens if $\mu_1$ is suitably large. Moreover, thanks to Lemma \ref{l3.5}, vanishing occurs if $\mu_1$ is small enough. Using Lemma \ref{l2.1} yields that $(u,v,h)$ is strictly increasing in $\mu_1$. Combining these with the continuous dependence of $(u,v,h)$ on $\mu_1$ and arguing as in the proof of \cite[Theorem 3.10]{WD}, we can finish the proof. The details are omitted here.
\end{proof}

Now we are in the position to investigate the effect of initial function $(u_0,v_0)$ on spreading and vanishing. Due to Remark \ref{r3.1}, we know that if $(u_0,v_0)$ is small enough, then vanishing happens. Our next result show that if one of $u_0$ and $v_0$ is sufficiently large, then spreading occurs.

\begin{lem}\label{l3.8}Suppose $h_0<l_0$. Spreading will happen if $u_0$ or $v_0$ is large enough.
\end{lem}
\begin{proof}
We will prove this result by constructing some suitable lower solution for \eqref{1.6}. The proof is divided into two step. Since the arguments are parallel, we only prove this assertion for $u_0$.

{\bf Step 1.} In this step, we deal with the case operator $\mathbb{B}[w]=w$. Clearly, the solution component $(u,h)$ satisfies
\bess\left\{\!\begin{array}{ll}
u_t\ge d_1u_{xx}-au, &t>0,~x\in(0,h(t)),\\[1mm]
u(t,0)=u(t,h(t))=0, &t>0,\\[1mm]
h'(t)>-\mu_1 u_x(t,h(t)), & t>0,\\[1mm]
h(0)=h_0, ~ u(0,x)=u_{0}(x),&0\le x\le h_0.
 \end{array}\right.
 \eess
It is easy to verify that the eigenvalue problem
\bes\left\{\!\begin{array}{ll}\label{3.7}
d_1\varphi''+\frac{1}{2}\varphi+\lambda\varphi=0, &x\in(\frac{1}{2},1),\\[1mm]
\varphi'(\frac{1}{2})=\varphi(1)=0
 \end{array}\right.
 \ees
has an eigenpair $(\lambda_1,\varphi)$ with $\lambda>\frac{1}{16d_1}$, $\varphi$ positive in $[\frac{1}{2},1)$ and $\varphi'<0$ in $(\frac{1}{2},1]$. Define
 \bess
 \Phi=\varphi(1-x), ~ ~ x\in[0,1/2]; ~ ~ ~ \Phi(x)=\varphi(x), ~ ~ x\in[1/2,1].
 \eess
 Certainly, $\Phi$ satisfies
 \bess\left\{\!\begin{array}{ll}
d_1\Phi''+\frac{{\rm sgn}(x-\frac{1}{2})}{2}\Phi'+\lambda_1\Phi=0, &x\in(0,1),\\[1mm]
\Phi(0)=\Phi(1)=0.
 \end{array}\right.
 \eess
Let positive constants $\ep,L_0,T,\sigma,\rho$ satisfy
\bess
0<\ep<\min\{1,h^2_0\}, ~ ~ L_0=1+l_0, ~ ~ T>L^2_0,\\
\sigma>\lambda_1+a(T+1), ~ ~ -2\mu_1\rho\Phi'(1)>(T+1)^{\sigma}.
\eess
Define
\bess
\ud h(t)=\sqrt{t+\ep} ~ ~ {\rm and } ~ ~
\ud u=\frac{\rho}{(t+\ep)^{\sigma}}\Phi(\frac{x}{\ud h}).
\eess
Direct computations yield that for $t\in(0,T]$ and $x\in(0,\ud h(t))$,
\bess
\ud u_t-d_1\ud u_{xx}+a\ud u&=&\frac{-\rho}{(t+\ep)^{\sigma+1}}\kk(\sigma\Phi+\frac{x\Phi'}{2\ud h}+d_1\Phi''-a(t+\ep)\Phi\rr)\\
&\le&\frac{-\rho}{(t+\ep)^{\sigma+1}}\kk(\frac{{\rm sgn}(x-\frac{1}{2})\Phi'}{2}+d_1\Phi''+[\sigma-a(t+\ep)]\Phi\rr)\\
&\le&\frac{-\rho}{(t+\ep)^{\sigma+1}}\kk(\frac{{\rm sgn}(x-\frac{1}{2})\Phi'}{2}+d_1\Phi''+\lambda_1\Phi\rr)=0.
\eess
Notice that $\ud h(0)=\sqrt{\ep}<h_0$. Thus $\ud u(0,x)=\frac{\rho}{\ep^{\sigma}}\Phi(\frac{x}{\sqrt{\ep}})\le u_0(x)$ in $[0,\sqrt{\ep}]$ if $u_0$ is large enough. Moreover, $\ud u(t,0)=\ud u(t,\ud h)=0$ for $t\in[0,T]$. Simple calculations show
\bess
\ud h'+\mu_1 \ud u_x(t,\ud h)=\frac{1}{2\ud h}+\mu_1\frac{\rho}{(t+\ep)^{\sigma}\ud h}\Phi'(1)\le \frac{1}{2\ud h}\kk(1+\frac{2\mu_1\rho\Phi'(1)}{(T+1)^{\sigma}}\rr)<0.
\eess
By a comparison argument, we know that $h(t)\ge\ud h(t)$ for $t\in[0,T]$. Thus $h(T)\ge\ud h(T)=\sqrt{T+\ep}>L_0>l_0$, which combined with Lemma \ref{l3.4} implies that spreading happens. This step is ended.

{\bf Step 2.} This step involves the case $\mathbb{B}[w]=w'$. Let $(\lambda_1,\varphi_1)$ be the principal eigenpair of \eqref{3.7} with $1/2$ replaced by $0$. Clearly, $\lambda_1>\frac{1}{16d_1}$, $\varphi_1$ is positive in $[0,1)$ and $\varphi'<0$ in $(0,1]$.
The positive constants $\ep,L_0,T,\sigma,\rho$ are given as above but $-2\mu_1\rho\varphi'_1(1)>(T+1)^{\sigma}$.
Define
\bess
\ud h(t)=\sqrt{t+\ep} ~ ~ {\rm and } ~ ~
\ud u=\frac{\rho}{(t+\ep)^{\sigma}}\varphi_1\kk(\frac{x}{\ud h}\rr).
\eess
Direct computations yield that for $t\in(0,T]$ and $x\in(0,\ud h(t))$,
\bess
\ud u_t-d_1\ud u_{xx}+a\ud u
&=&\frac{-\rho}{(t+\ep)^{\sigma+1}}\kk(\sigma\varphi_1+\frac{x\varphi'_1}{2\ud h}+d_1\varphi''_1-a(t+\ep)\varphi_1\rr)\\
&\le&\frac{-\rho}{(t+\ep)^{\sigma+1}}\kk(\frac{\varphi'_1}{2}+d_1\varphi''_1
+[\sigma-a(t+\ep)]\varphi_1\rr)\\
&\le&\frac{-\rho}{(t+\ep)^{\sigma+1}}\kk(\frac{\varphi'_1}{2}+d_1\varphi''_1
+\lambda_1\varphi_1\rr)=0.
\eess
Obviously, $\ud h(0)=\sqrt{\ep}<h_0$. So we can let $u_0$ be large enough such that $\ud u(0,x)=\frac{\rho}{\ep^{\sigma}}\Phi(\frac{x}{\sqrt{\ep}})\le u_0(x)$ in $[0,\sqrt{\ep}]$.  Besides, $\ud u(t,0)=\ud u(t,\ud h)=0$ for $t\in[0,T]$. Similar to step 1, we have
\bess
\ud h'+\mu_1 \ud u_x(t,\ud h)=\frac{1}{2\ud h}+\mu_1\frac{\rho}{(t+\ep)^{\sigma}\ud h}\varphi'_1(1)\le \frac{1}{2\ud h}\kk(1+\frac{2\mu_1\rho\varphi'_1(1)}{(T+1)^{\sigma}}\rr)<0.
\eess
Analogously, the desired result follows from a comparison method. The proof is finished.
\end{proof}

Let us parameterize the initial function $(u_0,v_0)$ as $\tau(\vartheta_1,\vartheta_2)$ with $\tau>0$ and {\bf (I)} holding for $(\vartheta_1,\vartheta_2)$. With the aid of Remark \ref{r3.1} and Lemma \ref{l3.8}, we can derive a critical value for $\tau$ governing spreading and vanishing.

\begin{lem}\label{l3.9} Assume that $h_0<l_0$. Let $(u_0,v_0)=\tau(\vartheta_1,\vartheta_2)$. Then there exists a unique $\tau^*>0$ such that spreading happens if and only if $\tau>\tau^*$.
\end{lem}
\begin{proof}
Due to Remark \ref{r3.1}, vanishing occurs for all small $\tau$. In view of Lemma \ref{l3.8}, spreading happens for all large $\tau>0$. From Lemma \ref{l2.1}, it follows that the unique solution $(u,v,h)$ of \eqref{1.6} is strictly increasing in $\tau$. Define
\[\tau^*=\inf\{\tau_0>0: {\rm spreading ~ happens ~ for ~ }\tau\ge\tau_0\}.\]
Clearly, $\tau^*$ is well defined and $\tau^*>0$. It is easy to see that spreading happens for all $\tau>\tau^*$. Now we show that vanishing occurs if $\tau<\tau^*$. Otherwise, there exists a $\tau_0\in (0,\tau^*)$ such that spreading happens with $\tau=\tau_0$. By the monotonicity, we deduce that spreading occurs for all $\tau\ge\tau_0$ which obviously contradicts the definition of $\tau^*$. Thus if $\tau<\tau^*$, then vanishing must happen.

It remains to check the case $\tau=\tau^*$. If spreading happens for $\tau=\tau^*$, then there exists a $t_0>0$ such that $h(t_0)>l_0$. By the continuous dependence of $(u,v,h)$ on $\tau$, there is a small $\ep>0$ such that the unique solution $(u_{\ep},v_{\ep},h_{\ep})$ of \eqref{1.6} with $\tau=\tau^*-\ep$ satisfies that $h_{\ep}(t_0)>l_0$ which, by Lemma \ref{l3.4}, implies that spreading occurs. Clearly, this is a contradiction. Hence if $\tau=\tau^*$, then vanishing occurs. The proof is finished.
\end{proof}

Clearly, Theorem \ref{t1.3} follows from Lemmas \ref{l3.3}-\ref{l3.9}.


\begin{thebibliography}{99}
\bibliographystyle{siam}
\setlength{\baselineskip}{15pt}

\vspace{-1.5mm}\bibitem{HY} C.-H. Hsu and T.-S. Yang, {\it Existence, uniqueness, monotonicity and asymptotic behaviour of travelling waves for epidemic models}, Nonlinearity, \textbf{26} (2013), 121-139.

\vspace{-1.5mm}\bibitem{CP} V. Capasso and S.L. Paveri-Fontana, {\it A mathematical model for the 1973 cholera epidemic in the European Mediterranean region}, Rev. Epidemiol. Sante Publique, \textbf{27} (1979), 121-132.

\vspace{-1.5mm}\bibitem{DL}Y.H. Du and Z.G. Lin, {\it Spreading-vanishing dichotomy in the diffusive logistic model with a free boundary}, SIAM J. Math. Anal., \textbf{42}  (2010), 377-405.

\vspace{-1.5mm}\bibitem{ABL}I. Ahn, S. Beak and Z.G. Lin, {\it The spreading fronts of an infective environment in a man-environment-man epidemic model}, Appl. Math. Model., \textbf{40} (2016), 7082-7101.

\vspace{-1.5mm}\bibitem{ZLN}M. Zhao, W.-T. Li and W.J. Ni, {\it Spreading speed of a degenerate and cooperative epidemic model with free boundaries}, Discrete Contin. Dyn. Syst., Ser. B, \textbf{25} (2020), 981-999.

\vspace{-1.5mm}\bibitem{WD}R. Wang and Y.H. Du, {\it Long-time dynamics of a diffusive epidemic model with free boundaries}, Discrete Contin. Dyn. Syst., Ser. B, \textbf{26} (2021), 2201-2238.

\vspace{-1.5mm}\bibitem{KY} Y. Kaneko and Y. Yamada, {\it A free boundary problem for a reaction diffusion equation appearing in ecology}, Adv. Math. Sci. Appl., {\bf 21} (2011), 467-492.

\vspace{-1.5mm}\bibitem{Wjde14}M.X. Wang, {\it On some free boundary problems of the prey-predator model}, J. Differential Equations, {\bf 256} (2014), 3365-3394.

\vspace{-1.5mm}\bibitem{Wjde15} M.X. Wang, {\it The diffusive logistic equation with a free boundary and sign-changing coefficient}, J. Differential Equations, {\bf 258} (2015), 1252-1266.

\vspace{-1.5mm}\bibitem{WZrwa15}M.X. Wang and Y. Zhang, {\it Two kinds of free boundary problems for the diffusive prey-predator model}, Nonlinear Anal.
Real World Appl., \textbf{24} (2015), 73-82.

\vspace{-1.5mm}\bibitem{Wcnsns15} M.X. Wang, {\it Spreading and vanishing in the diffusive prey-predator model with a free boundary},  Commun. Nonlinear Sci. Numer. Simulat., {\bf 23} (2015), 311-327.

\vspace{-1.5mm}\bibitem{Wjfa16} M.X. Wang, {\it A diffusive logistic equation with a free boundary and sign-changing coefficient in time-periodic environment}, J. Funct. Anal., {\bf 270} (2016), 483-508.

\vspace{-1.5mm}\bibitem{ZWjdde18} Y.G. Zhao and M.X. Wang, {\it A reaction-diffusion-advection equation with mixed and free boundary conditions}, J. Dyn. Diff. Equat., {\bf 30} (2018), 743-777.

\vspace{-1.5mm}\bibitem{LZ}Z.G. Lin and H.P. Zhu, {\it Spatial spreading model and dynamics of West Nile virus in birds and mosquitoes with free boundary}, J. Math. Biol., \textbf{75} (2017), 1381-1409.

\vspace{-1.5mm}\bibitem{WND} Z.G. Wang, H. Nie and Y.H. Du, {\it Spreading speed for a West Nile virus model with free boundary}, J. Math. Biol., \textbf{79} (2019), 433-466.

\vspace{-1.5mm}\bibitem{HW2} H.M. Huang and M.X. Wang, {\it A nonlocal SIS epidemic problem with double free boundaries}, Z. Angew. Math. Phys., \textbf{70} (2019), 109.

\vspace{-1.5mm}\bibitem{CLWY}J.-F. Cao, W.-T. Li, J. Wang and F.Y. Yang, {\it A free boundary problem of a diffusive SIRS model with nonlinear incidence}, Z. Angew. Math. Phys., \textbf{68} (2017), 39.

\vspace{-1.5mm}\bibitem{ZRZ}D.D. Zhu, J.L. Ren and H.P. Zhu, {\it Spatial-temporal basic reproduction number and dynamics for a dengue disease diffusion model}, Math. Meth. Appl. Sci., \textbf{41} (2018), 5388-5403.

\vspace{-1.5mm}\bibitem{Du} Y.H. Du, {\it Propagation and reaction diffusion models with free boundaries}, Bull. Math. Sci., \textbf{12} (2022), 2230001.

\vspace{-1.5mm}\bibitem{DMZ}Y.H. Du, H. Matsuzawa and M.L. Zhou, {\it Sharp estimate of the spreading speed determined by nonlinear free boundary problems}, SIAM J. Math. Anal., \textbf{46} (2014), 375-396.

\vspace{-1.5mm}\bibitem{DLou}Y.H. Du and B.D. Lou, {\it Spreading and vanishing in nonlinear diffusion problems with free boundaries}, J. Eur. Math. Soc., \textbf{17} (2015), 2673-2724.

\vspace{-1.5mm}\bibitem{WND1}Z.G. Wang, H. Nie and Y.H. Du, {\it Sharp asymptotic profile of the solution to a West Nile virus model with free boundary}, European J. Appl. Math., (2023), https://doi.org/10.1017/S0956792523000281.

\vspace{-1.5mm}\bibitem{LLW}L. Li, S.Y. Liu and M.X. Wang, {\it A viral propagation model with a nonlinear infection rate and free boundaries},
Sci. China Math., \textbf{64} (2021), 1971-1992.

\vspace{-1.5mm}\bibitem{Troy} W.C. Troy, {\it Symmetry properties in systems of semilinear elliptic equations}, J. Differential Equations, {\bf 42} (1981), 400-413.

\vspace{-1.5mm}\bibitem{WangPara} M.X. Wang, {\it Nonlinear Second Order Parabolic Equations}, Boca Raton: CRC Press, 2021.

\vspace{-1.5mm}\bibitem{WZou}J.H. Wu and X.F. Zou, {\it Traveling wave fronts of reaction-diffusion systems with delay}, J. Dyn. Diff. Equat., {\bf 13} (2001), 651-687.

\vspace{-1.5mm}\bibitem{Duorder}Y.H. Du, {\it Order structure and topological methods in nonlinear partial differential equations: Maximum principles and Applications}, World Scientific, 2006.

\vspace{-1.5mm}\bibitem{ChenLi}W.X. Chen and C.M. Li, {\it Maximum principles for the fractional p-Laplacian and symmetry of solutions}, Adv. Math., \textbf{335} (2018), 735-758.

\vspace{-1.5mm}\bibitem{Dja}D.G. De Figueredo, {\it Monotonicity and symmetry of solutions of elliptic systems in general domains}, NODEA-Nonlinear Diff., \textbf{1} (1994), 119–123.

\vspace{-1.5mm}\bibitem{Wdcdsb2021}M. X. Wang, {\it Existence and uniqueness of solutions of free boundary problems in heterogeneous environments},
Discrete Contin. Dyn. Syst., Ser. B, \textbf{24} (2019), 415-421.
Erratum: Discrete Contin. Dyn. Syst., Ser. B, \textbf{27} (2021), 5179-5180.
\end{thebibliography}
\end{document}